\documentclass[12pt, a4paper, parskip=half, abstracton]{scrartcl}

\usepackage{array}
\usepackage{marginnote}
\usepackage{xcolor}
\usepackage{amscd,amssymb,amsfonts,amsmath,latexsym,amsthm}
\usepackage{hyperref}
\usepackage[all,cmtip]{xy}
\usepackage{mathrsfs}
\usepackage{graphicx}
\usepackage{bm} 
\usepackage{soul} 

\usepackage{etoolbox}

\def\hB{\hspace*{\fill}$\qed$}
\def\heB{\hspace*{\fill}$\blacksquare$}
 \def\hrB{\hspace*{\fill}$\spadesuit$}

\usepackage{defs_pp1}
\usepackage{slashed}
\usepackage[utf8]{inputenc}
\usepackage{microtype}
\usepackage[english]{babel}

\title{Localization for coarse homology theories}
\author{
Ulrich Bunke\thanks{Fakult{\"a}t f{\"u}r Mathematik,
Universit{\"a}t Regensburg,
93040 Regensburg,
GERMANY\newline
ulrich.bunke@mathematik.uni-regensburg.de} 
 \and
Luigi Caputi  \thanks{Fakult{\"a}t f{\"u}r Mathematik,
Universit{\"a}t Regensburg,
93040 Regensburg,
GERMANY\newline
luigi.caputi@unibo.it
}
}

\numberwithin{equation}{section}
\newtheorem{theorem}{Theorem}[section] 
\newtheorem{prop}[theorem]{Proposition}
\newtheorem{lem}[theorem]{Lemma}

\newtheorem{ddd}[theorem]{Definition}
\newtheorem{kor}[theorem]{Corollary}

\newtheorem{prob}[theorem]{Problem}

\theoremstyle{remark}
\theoremstyle{definition}

\newtheorem{ex}[theorem]{Example}
\newtheorem{rem}[theorem]{Remark}

\newcommand{\nCalg}{C^{*}\mathbf{Alg}^{\mathrm{nu}}}
\newcommand{\KKG}{\mathrm{KK}^{G}}
\newcommand{\KKH}{\mathrm{KK}^{H}}
\newcommand{\KK}{\mathrm{KK}}
\newcommand{\kkG}{\mathrm{kk}^{G}}

\newcommand{\cFin}{\mathcal{F}in}
\newcommand{\Tw}{\mathbf{Tw}}
\newcommand{\Bredon}{\mathrm{Bredon}}

\newcommand{\All}{\mathcal{A}\ell\ell}

\newcommand{\UBC}{\mathbf{UBC}}

\newcommand{\Coind}{\mathrm{Coind}}

\renewcommand{\Ind}{\mathrm{Ind}}

\newcommand{\Yo}{\mathrm{Yo}}

\newcommand{\Res}{\mathrm{Res}}

\newcommand{\Orb}{\mathbf{Orb}}

\newcommand{\Hilb}{\mathbf{Hilb}}

\newcommand{\bQ}{\mathbf{Q}}

\newcommand{\BC}{\mathbf{BornCoarse}}

\newcommand{\Fin}{\mathcal{F}\mathit{in}}

\newcommand{\PSh}{{\mathbf{PSh}}}

\newcommand{\Add}{{\mathtt{Add}}}

\newcommand{\bA}{{\mathbf{A}}}

\newcommand{\cO}{{\mathcal{O}}}

\newcommand{\cY}{{\mathcal{Y}}}

\newcommand{\Spc}{\mathbf{Spc}}

\newcommand{\Ccat}{{\mathbf{C}^{\ast}\mathbf{Cat}}}
\newcommand{\Calg}{{\mathbf{C}^{\ast}\mathbf{Alg}}}

\renewcommand{\Add}{\mathbf{Add}}

\renewcommand{\colim}{\operatorname{colim}}\renewcommand{\lim}{\operatorname{lim}}

\begin{document}

\maketitle
\begin{abstract}
We  introduce the notion of a Bredon-style equivariant coarse homology theory. We  show that such a   Bredon-style equivariant coarse homology theory satisfies localization theorems {and that }
a general  equivariant coarse homology theory can be approximated by a Bredon-style version.  We discuss the special case of algebraic and topological equivariant coarse $K$-homology and obtain
the coarse analog of 
Segal's localization theorem.  
\end{abstract}

\tableofcontents
\section{Introduction}

This paper contributes to the study of general properties of equivariant coarse homology theories. In particular, we are interested in   localization theorems 
relating the value of a{n equivariant coarse homology} theory $E$ on a $G$-bornological coarse space with  the values on certain  fixed point  {subspaces.} 
These localization results are the analogues, for equivariant coarse homology theories, of the classical localization theorem  for equivariant K-cohomology theory by Segal \cite[{Prop.~4.1}]{segalk}.

In the present paper we use the language of  $\infty$-categories as developed by \cite{cisin}, \cite{Joyal} \cite{htt}.
For a group $G$,
the category  {$G\BC$} of $G$-bornological coarse spaces and the notion of an equivariant coarse homology theory  have been introduced in  \cite{equicoarse}. We will freely use the language developed therein.

  As a preparation for our localization results  we {first} show the Abstract Localization Theorem~I~\ref{wregfio324r3t3t4t} whose statement we {now} 
  explain. 
  Let  $G\Orb$ be the orbit category  of a group  {$G$}. For functors $E\colon G\Orb\to \Fun^{\colim}(\bC,\bD)$ and $X\colon G\Orb^{op}\to \bC$, where $\bC$ and~$\bD$ are cocomplete $\infty$-categories {and $\Fun^{\colim}$ denotes the $\infty$-category of colimit-preserving functors},  we   define the $\bD$-valued  pairing
as the coend
$$E^{G}(X):=\int^{G\Orb} E\circ X\ .$$
If $\cF$ is a  family of subgroups of $G$, $E$ vanishes on  $\cF$ {(see Definition \ref{fkjkjafsljashfgfasljh})}, and $X^{\cF}$
is the extension by zero of the restriction $X_{|\cF^{\perp}}$ of $X$ to the complement $\cF^{\perp}$ of $\cF$ {(in the family of all subgroups of $G$)}, then  the Abstract Localization Theorem I  \ref{wregfio324r3t3t4t} asserts that the morphism
$$E^{G}(X^{\cF})\to E^{G}(X)$$ induced by the canonical morphism $X^{\cF}\to X$  is an equivalence.  
    
For illustration, we first apply this theorem in the case of equivariant homotopy theory {(see Section \ref{hdagghagghag})}.
In this case  $E\colon G\Orb\to  \bC$ represents a  {homology theory with values in a}   cocomplete stable $\infty$-category $\bC$, $X\colon G\Orb^{op}\to \Spc$
is the equivariant homotopy type of a   $G$-topological space, and $E^{G}(X)$ is the evaluation of the {(associated)} equivariant homology theory {$E^G$} 
on $X$. Here, in order to define $E^{G}$, we use the equivalence $\Fun^{\colim}(\Spc,\bC)\simeq \bC$  given by {the} universal property of the $\infty$-category of spaces $\Spc$ and cocompleteness of~$\bC$.
A conjugacy class $\gamma$ in $G$ gives rise to the  family  $\cF(\gamma)$ of  all subgroups   of $G$ which intersect  $\gamma$ trivially, see Definition \ref{ioujiowrgrgrergergr}. The  
Abstract Localization Theorem II \ref{wriowwgwegewgewgewgw} asserts that (under a condition on the quality of the inclusion of the fixed-point set $X^{\gamma}\to X$) 
the {induced} morphism
$$E^{G}(X^{\gamma})\to  E^{G}(X)$$
is an equivalence provided $E$ vanishes on $\cF(\gamma)$.
This theorem is deduced from   the Abstract Localization Theorem I using excision.
 In Section \ref{giuhiuhuir32r23r232} we derive   the homological version of the  classical
 Segal Localization Theorem for equivariant topological $K$-homology $KU_{G}$, assuming that $G$ is a  finite group. The main point is to show that the functor
$KU_{G,(\gamma)}$ (the localization of $KU_{G}$ at $(\gamma)$) vanishes on $\cF(\gamma)$, where $(\gamma)$ is the ideal of the representation ring $R(G)$ consisting of elements whose trace vanishes on $\gamma$.

The main topic  of the present paper {is the study of} localization theorems for equivariant coarse homology theories. Recall that
in equivariant homotopy theory, as a consequence of {Elmendorf's} theorem, $\bC$-valued equivariant homology theories are the same as functors $E\colon G\Orb\to \bC$.  {Under the} identification 
$\bC\simeq \Fun^{\colim}(\Spc,\bC)$, 
we can interpret $E$ as a functor from $G\Orb$ with values in non-equivariant homology theories. 
Since we do not have a coarse version of {Elmendorf's} theorem we will therefore 
work with the latter picture of  functors from $G\Orb$ to  non-equivariant coarse homology theories.

The natural replacement of the functor $\Top\to \Spc$ in coarse homotopy theory is the functor $\Yo\colon \BC\to \Spc\cX$ (see~{\eqref{344frf}}) which associates to a bornological coarse space its motivic coarse  space.
We shall see that a functor $E\colon G\Orb\to \Fun^{\colim}(\Spc\cX,\bC)$ from $G\Orb$ to non-equivariant coarse homology theories gives rise to a $\bC$-valued equivariant coarse homology theory $$X\mapsto E^{G}(X):=\int^{{G\Orb}}E\circ \tilde Y(X)\ ,$$ (see Corollary \ref{gih42iuhjfierfwfef}).
Here the functor $\tilde Y\colon G\BC\to \Fun(G\Orb^{op},\Spc\cX)$ sends a $G$-bornological coarse space $X$ to the functor on the orbit category which sends $G/H$ to the motive $\Yo(X^{H})$ of the bornological coarse space of $H$-fixed points in $X$. Equivariant coarse homology theories of this form will be {called of}  Bredon-style.

The Abstract Localization Theorem I now implies a Coarse Abstract Localization Theorem~I~\ref{fueiwfhuiwefewf23r23f} asserting that for a family of subgroups $\cF$
the natural morphism $$E^{G}(X^{\cF})\to E^{G}(X)$$
is an equivalence provided $E$ vanishes {on} $\cF$. Using excision we deduce the 
Coarse Abstract Localization Theorem II \ref{rgioewfwefewfewfewf} which says that for a conjugacy class $\gamma$ of $G$ the natural morphism
$$E^{G}(X^{\gamma})\to E^{G}(X)$$
is an equivalence provided $E$ vanishes on ${\cF(\gamma)}$  and the inclusion $X^{\gamma}\to X$
satisfies a natural regularity condition.

Since in coarse homotopy theory we are lacking the analogue of {Elmendorf's} theorem we do not expect that a general  equivariant coarse homology theory is  of  Bredon-style. In Definition \ref{rgfrihgioffwfwfwefewf} we associate to every equivariant coarse homology theory a Bredon-style approximation $$E^{\Bredon}\to E\ .$$ To this end we consider the functor
$$\underline{E}:G\Orb\to \Fun^{\colim}(\Spc\cX,\bC)$$ sending $S$ to
the functor
$X\mapsto E (S_{min,max}\otimes X)$, and we  set $$E^{\Bredon}:=\underline{E}^{G}\circ \tilde Y\ .$$
As an immediate corollary of the  Coarse Abstract Localization Theorem II \ref{rgioewfwefewfewfewf}  we deduce in Corollary \ref{gerhgiu34hiu34huferererg}
that, for a conjugacy class $\gamma$ of $G${, the  morphism}
$$E^{\Bredon}(X^{\gamma})\to E^{\Bredon}(X)$$ induced by the inclusion $X^{\gamma}\to X$ is an equivalence provided the functor
$\underline{E}$ vanishes on~$\cF(\gamma)$.
In order to be able to deduce   a localization theorem for $E$ itself, in Section \ref{f42oifj2o3f32f32e} we study conditions on $E$ and $X$ implying that {the comparison map} $E^{\Bredon}(X)\to E(X)$ is an equivalence. {We also} provide some examples which show that this comparison map is not always an equivalence.
 Since the formulations are technical we refrain from stating these results here in the introduction.

In the Sections \ref{greiuhiufhiwefwefewfewf2r323r} and \ref{greiuhiufhiwefwefewfewf2r323r1} we show how the Abstract Coarse Localization Theorems can be applied in the case of topological (or algebraic) equivariant coarse $K$-homology theories $K\cX^{G}$ (or $K\bA\cX^{G}$ for some additive category $\bA$ enriched in $\C$-vector spaces) and finite groups $G$.  We first observe that the {values of these functors}  can be localized at an ideal $(\gamma)$ of $R(G)$ associated to a    conjugacy class $(\gamma)$, and
then show that for  every object $T$ in $G\Sp\cX$  the {$T$-twisted} functor
$\underline{K\cX^{G}_{T,(\gamma)}}$ (or $\underline{K\bA\cX^{G}_{T,(\gamma)}}$, respectively) vanishes on $F(\gamma)$.
Under additional conditions on $X$ and $X^{\gamma}$ (ensuring that the approximation by the Bredon-style versions {works}) we then deduce equivalences
$$K\cX^{G}_{T,(\gamma)}(X^{\gamma})\stackrel{\simeq}{\to} K\cX^{G}_{T,(\gamma)}(X)$$ and
$$K\bA\cX^{G}_{T,(\gamma)}(X^{\gamma})\stackrel{\simeq}{\to} K\bA\cX^{G}_{T,(\gamma)}(X)\ .$$ 

Since a Bredon-style equivariant coarse homology theory is determined by a functor defined on the orbit category it is natural to consider assembly maps.
 An  assembly map  approximates a homology theory by the theory obtained from the restriction of the initial functor to the subcategory of the orbit category of transitive $G$-sets whose stabilizers belong to a given family of subgroups. In equivariant homotopy theory these assembly maps are the main objects of the isomorphism conjectures like the Farell-Jones or Baum-Connes conjectures. An equivariant coarse homology theory naturally gives rise to an equivariant homology theory in the sense of equivariant homotopy theory. In  
Section \ref{greoijio3gergrgregergerg} we show that the fact that the assembly map for this derived equivariant homology theory is an equivalence implies that the assembly map for the equivariant coarse homology theory is an equivalence, at least after restriction to a certain subcategory of $G\BC$.

 Acknowledgements:
 {\em U.B.~was supported by the SFB  1085 (Higher Invariants) and L.C.~was supported by the GK 1692 (Curvature, Cycles, and Cohomology). L.C.~also acknowledges partial support from INdAM-GNSAGA.}

\section{Localization theorems in equivariant topology}

The aim of this section is to provide abstract localization theorems, and to derive from them the classical localization result of Segal.

\subsection{The Abstract Localization Theorem for presheaves on the orbit category}

 Let $G$ be a group and $G\Orb$ be the category of transitive $G$-sets and equivariant maps.
 Assume that $\bC$ is a cocomplete $\infty$-category.
 By the universal property of presheaves pull-back along  the Yoneda embedding 
$i\colon G\Orb\to \PSh(G\Orb)$ induces an equivalence \begin{equation}\label{3fkjkjnjkdnkjwdwdqwdqwd}
i^{*}\colon \Fun^{\colim}(\PSh(G\Orb),\bC) \stackrel{\simeq}{\to}  \Fun(G\Orb,\bC)
\end{equation}
of $\infty$-categories \cite[Thm.~5.1.5.6]{htt}.

We consider  a functor 
  $E\colon G\Orb\to \bC$.
  
  \begin{ddd}\label{uiweogwerwwefwef} By   $$E^{G}\colon \PSh(G\Orb)\to \bC$$   we denote the essentially uniquely determined colimit-preserving functor   with  $i^{*}E^{G}\simeq E$. \end{ddd}
  
In the following, we use the language of ends and coends, see e.g.~\cite{MR3690268}, \cite{Glasman2016}.
{Denote by $\Spc$ the $\infty$-category of spaces. }
The functor $E^{G}$ is given by the  left Kan extension of $E$  along the Yoneda embedding and can  
	 be expressed in terms of a coend  \begin{equation}\label{verlkjeolrvevevevv}
		E^{G}(X)\simeq \int^{G\Orb} E\otimes X\ ,
	\end{equation} where $\otimes\colon\bC\otimes \Spc\to \bC $ is the tensor structure of $\bC$ over spaces and we consider $E\otimes X$ as a functor from
	$G\Orb\times G\Orb^{op}$ to $\bC$.

%
%
 Let $\cF$ be a 
 conjugation invariant set of subgroups of $G$. It determines the full subcategory $G_{\cF}\Orb$ of $G\Orb$ of transitive $G$-sets with stabilizers in $\cF$. The inclusion of
$G_{\cF}\Orb$ into $G\Orb$ induces   adjunctions \begin{equation}\label{fwrelkoelfwefwef}
\Ind_{\cF}:\Fun(G_{\cF}\Orb^{op},\bD)\leftrightarrows \Fun(G\Orb^{op},\bD):\Res_{\cF}
\end{equation}
(assuming that  $\bD$ is cocomplete)
and $$\Res_{\cF}:\Fun(G \Orb^{op},\bD)\leftrightarrows \Fun(G_{\cF}\Orb^{op},\bD):\Coind_{\cF}$$
(assuming that $\bD$ is complete). The left adjoint $\Ind_{\cF}$ (resp.~right-adjoint $\Coind_{\cF}$)  is a left- (resp.~right-) Kan extension 
functor.
We have similar adjunctions for covariant functors which will be denoted by the same {symbols.}

Let  $\cF$ be a conjugation invariant set  of subgroups  of $G$.
\begin{ddd}We let $\cF^{\perp}:=\All\setminus \cF$ denote the complement of $\cF$ in the set $\All$ of all subgroups of $G$.
\end{ddd}

Let  $X$ be in $ \Fun(G \Orb^{op},\bD)$ with  $\bD$ cocomplete.
 \begin{ddd}\label{rgijofewrgergergerg}
 We define \begin{equation}\label{g45goijkfml3rf3f3f}
X^{\cF}:=\Ind_{\cF^{\perp}}\Res_{\cF^{\perp}} X
\end{equation}  and let $X^{\cF}\to X$ denote the map induced by the counit of the adjunction \eqref{fwrelkoelfwefwef}.
 \end{ddd}

Let $\bC$ be a  pointed {$\infty$-category} and 
 consider a   functor $E\colon G\Orb\to \bC$.
\begin{ddd}\label{fkjkjafsljashfgfasljh}
We say that $E$ vanishes on   $\cF$ if  $E(S)\simeq 0$ for all $S$ in $G_{\cF}\Orb$.
\end{ddd}

 Let $\cF$ be a 
 conjugation invariant set of subgroups of $G$.
 
\begin{ddd} $\cF$   is called a family of subgroups if it is invariant under taking subgroups. \end{ddd}

We can now formulate the abstract localization theorem.
 Let  $E\colon G\Orb\to \bC$ be a functor with target a stable   cocomplete $\infty$-category $\bC$, $\cF$ be a  conjugation invariant set   of subgroups of $G$, and
 $X$ be in $\PSh(G\Orb)$.
 \begin{theorem}[Abstract Localization Theorem I]\label{wregfio324r3t3t4t} Assume:
  \begin{enumerate}\item
 $E$ vanishes on $\cF$.
 \item $\cF$ is a family of subgroups of $G$.
 \end{enumerate}
Then, the morphism $X^{\cF}\to X$ induces an equivalence
$E^{G}(X^{\cF})\to E^{G}(X)$.
\end{theorem}
\begin{proof}
{We first show that} 
  the counit
$\Ind_{\cF^{\perp}}\Res_{\cF^{\perp}}E\to E$ of  the 
adjunction \eqref{fwrelkoelfwefwef} is an equivalence.
Since $G_{\cF^{\perp}}\Orb\to G\Orb$ is   fully faithful, the unit of the adjunction \eqref{fwrelkoelfwefwef} is an  equivalence $\id\stackrel{\simeq}{\to}\Res_{\cF^{\perp}}\Ind_{\cF^{\perp}} $. We conclude that if   
 $S$ belongs to $G_{\cF^{\perp}}\Orb$, then $(\Ind_{\cF^{\perp}}\Res_{\cF^{\perp}}E)(S)\to E(S)$ is an equivalence.
 Assume now that $S$ is in $G_{\cF}\Orb$.   Then $E(S)\simeq 0$ by {assumption}.
 By the object-wise formula for the left Kan extension functor $\Ind_{\cF^{\perp} }$ {we} have an equivalence  
 $$(\Ind_{\cF^{\perp}}\Res_{\cF^{\perp}}E)(S)\simeq \colim_{(T\to S)\in G_{\cF^{\perp}}\Orb  /S} E(T)\ .$$
 We now note\footnote{It is here where we use the special structure of the orbit category and the family. The remaining arguments would work for $G_{\cF}\Orb\subseteq G\Orb$ replaced by any pair of an $\infty$-category and a full subcategory.} that the existence of a morphism $T\to S$ implies that $T\in G_{\cF}\Orb$. Hence the index category of this colimit is empty and $(\Ind_{\cF^{\perp}}\Res_{\cF^{\perp}}E)(S)\simeq 0$.

In order to finish the proof of {the theorem} 
we use
the equivalence of  transformations \begin{equation}\label{ewrgergrefwf} E^{G}(\Ind_{\cF^{\perp}}\Res_{\cF^{\perp}}\to \id)\simeq  (\Ind_{\cF^{\perp}}\Res_{\cF^{\perp}}E\to E)^{G}\ .\qedhere 
\end{equation}

\end{proof}

\begin{ex}\label{gio2r23r23r32r2}
Let $\cF$ be a family of subgroups of $G$ and  $E\colon G\Orb\to \bC$ be a functor. Then we can form a new functor
$$E_{\cF^{\perp}}:=\Ind_{\cF^{\perp}}\Res_{\cF^{\perp}}E \ .$$
It comes with a natural morphism $E_{\cF^{\perp}}\to E$ and is the best approximation of $E$ by a functor which vanishes on
 $\cF$. \heB
\end{ex}

\subsection{The Abstract  Localization Theorem for topological $G$-spaces}
\label{hdagghagghag}

We now formulate a second version of the Abstract Localization Theorem for families of subgroups determined by a conjugacy class  $\gamma$ of $G$ and $G$-topological spaces.

\begin{ddd}\label{ioujiowrgrgrergergr}We define the set of subgroups of $G$ $$\cF(\gamma):=\{H\:\mbox{a subgroup of $G$}\:|\: H\cap \gamma=\emptyset\}\ .$$   \end{ddd}
One easily checks:

\begin{lem} \label{rgiorgergergg} $\cF(\gamma)$ is a family of subgroups.\end{lem}

Let $\ell\colon \Top\to \Spc$ be the functor which sends a topological space to the associated space. We consider the functor
\begin{equation}\label{v4toi4hgkjnetkvevv}
\tilde Y\colon G\Top\to \PSh(G\Orb)\ , \quad  X\mapsto (S\mapsto \ell(\Map_{G\Top}(S_{discr},X)))\ ,
\end{equation} 
where $\Map_{G\Top}$ denotes the  topological space of $G$-invariant maps with the compact-open topology.

\begin{rem}\label{ertefsfa}
The functor $\ell$  is equivalent to the localization functor $\Top\to \Top[W^{-1}]$, where $W$ is the class of weak homotopy equivalences.

The weak equivalences $W_{G}$ in equivariant homotopy theory are  the morphisms in $G\Top$ which induce weak homotopy equivalences on the fixed points spaces for all subgroups~$H$ of~$G$. These are exactly the morphisms which are sent to equivalences by $\tilde Y$. By Elmendorf's theorem,  the functor $\tilde Y$ induces an equivalence
  \begin{equation}\label{rvelkm3l4kgferveve}
G\Top[W_{G}^{-1}]\stackrel{\simeq}{\to} \PSh(G\Orb)
\end{equation}  and is therefore equivalent to the localization functor
$G\Top\to G\Top[W_{G}^{-1}]$.  See e.g.~\cite[Thm.~1.3.8]{blumberg} for this {formulation} of the classical result of  {Elmendorf}~\cite{MR690052}. \hrB
  \end{rem}
 
Let $X$ be a  
 $G$-topological space {and $\gamma$ be a conjugacy class of $G$}.
 
 \begin{ddd}We  define the $G$-topological space  $X^{\gamma}$ of    $\gamma$-fixed points to be the  subset
$  \bigcup_{g\in \gamma} X^{g}$  of $X$ with the induced topological structure and $G$-action.
\end{ddd}
%
 
%
%

\begin{lem}\label{3rgorg34g34g1}
We have
$\tilde Y(X)^{\cF(\gamma)}\simeq \emptyset$ if and only if $X^{\gamma}=\emptyset$.
\end{lem}
\begin{proof}
Assume that $X^{\gamma}=\emptyset$. If $H$ is in $\cF(\gamma)^{\perp}$, then there exists an element $h$ in $H\cap \gamma$. But then we have the inclusions
$X^{H}\subseteq  X^{h}\subseteq X^{\gamma}$, i.e., $ X^{H}=\emptyset$. This implies that $\tilde Y(X)(G/H)\simeq \emptyset$. 

We conclude that
$\Res_{\cF(\gamma)^{\perp}}\tilde Y(X)\simeq \emptyset$ and hence $X^{\cF(\gamma)}\simeq \emptyset$.

Assume now that $\tilde Y(X)^{\cF(\gamma)}\simeq \emptyset$.
Then necessarily $X^{H}=\emptyset$ for all $H$ in $\cF(\gamma)^{\perp}$.
For $h$ in $\gamma$ we have $\langle h\rangle \in \cF(\gamma)^{\perp}$ and hence
$X^{h}=X^{\langle h\rangle}=\emptyset$.

We conclude that $X^{\gamma}=\emptyset$.\end{proof}

  Recall that $\bC$ denotes a   cocomplete   
 stable $\infty$-category, and
   let $E\colon G\Orb\to \bC$ be a functor. For a $G$-topological space $X$ we
  will   use the abbreviation $E^{G}(X):=E^{G}(\tilde Y(X))$.  

\begin{prop}\label{foief223ff2f}
Assume: \begin{enumerate} \item $X^{\gamma}=\emptyset$
\item $E$ vanishes on $\cF(\gamma)$.
\end{enumerate}
Then $E^{G}(X) \simeq 0$.
\end{prop}
\begin{proof}
By Lemma \ref{rgiorgergergg} we know that $\cF(\gamma)$ is a family of subgroups of $G$.
Since $X^{\gamma}=\emptyset$ we have  
$\tilde Y(X)^{\cF(\gamma)}\simeq \emptyset$ by Lemma \ref{3rgorg34g34g1}. 
By  Theorem \ref{wregfio324r3t3t4t}
we have the  first equivalence in
$$ E^{G}( X )=E^{G}(\tilde Y(X)) \simeq E^{G}( \tilde Y(X)^{\cF(\gamma)}) \simeq 0\ . $$
 \end{proof}


Let  $\gamma$ be a conjugacy class in $G$,  let $E\colon G\Orb\to \bC$ be a functor, and  let $X$ be a $G$-topological space. Then Theorem \ref{wregfio324r3t3t4t}
has the following consequence.

 \begin{theorem}[Abstract Localization Theorem II]\label{wriowwgwegewgewgewgw}
Assume:\begin{enumerate}\item\label{gjkl} $X^{\gamma}$ is closed in $X$ and admits a 
open invariant neighbourhood $N$ such that $X^{\gamma}\to N$ is an equivariant weak equivalence (see Remark \ref{ertefsfa}). 
\item $E$ vanishes on $\cF(\gamma)$.
\end{enumerate}
 Then the inclusion
$X^{\gamma}\to X$ induces an equivalence
$$E^{G}( X^{\gamma} ) \to E^{G}( X) \ .$$
\end{theorem}
\begin{proof}


The functor $\ell \colon \Top\to \Spc$ sends push-out squares derived from open decompositions into push-out squares. This implies that $\tilde Y$ has the same {property.} 
 We apply this fact to the square 	
$$\xymatrix{N \setminus X^{\gamma}\ar[r]\ar[d]&N \ar[d]\\X\setminus X^{\gamma}\ar[r]&X}$$
for $N$  an open invariant neighbourhood of $X^{\gamma}$  such that $X^{\gamma}\to N$ is an equivariant weak equivalence.  The functor $\tilde Y$ sends it to a push-out square in $\PSh(G\Orb)$. The functor $E^{G}(-) $ preserves all colimits, hence push-out squares in particular.

Since
$(X\setminus X^{\gamma})^{\gamma}=\emptyset$ and
$(N \setminus X^{\gamma})^{\gamma}=\emptyset$, by 
Proposition \ref{foief223ff2f} the functor $E^{G}(-) $ sends the 
left part of the square to zero.    Hence  the morphism $E^{G}( N  ) \to E^{G}(X) $ is an equivalence. Since    $E^{G}(X^{\gamma})\to E^{G}(N)$ is an equivalence  by  Assumption \ref{wriowwgwegewgewgewgw}.\ref{gjkl},
 this implies that $E^{G}(X^{\gamma})\to E^{G}( X)$ is an equivalence.
%
 \end{proof}

\subsection{The  Segal localization theorem revisited}\label{giuhiuhuir32r23r232}

In this section we show how the classical Atiyah-Segal localization theorem for equivariant $K$-homology can be deduced from the {Abstract} Localization Theorem II \ref{wriowwgwegewgewgewgw}.

In this section we assume that $G$ is finite. We consider 
 {equivariant} topological  $K$- homology  as an 
 equivariant  homology theory with values in the category  $\Mod(R_{G})$ determined by   the commutative ring-spectrum  $R_{G}:=KU_{G}(*)$. 
   It is thus determined by a functor
 $$KU_{G}\colon G\Orb\to \Mod(R_{G})\ .$$ 
   It is known that  the ring $\pi_{0}R_{G}$ is canonically isomorphic to the representation ring $R(G)$ of~$G$.

 \begin{rem}\label{erthopetrghertge}
 A quick construction of the functor $KU_{G}$ can be given using the   spectrum-valued  equivariant $\mathrm{KK}$-theory from \cite{KKG}. Let $G\nCalg$  denote the symmetric monoidal category of possibly non-unital 
 $C^{*}$-algebras with $G$-action and the symmetric monoidal structure given by maximal tensor product. 
 In \cite{KKG}   a lax symmetric monoidal functor  $\kkG\colon G\nCalg\to \KKG$ to a presentably symmetric monoidal  stable $\infty$-category is constructed which is essentially
 uniquely characterized by universal properties which will not  be repeated here. 
 We write $\KKG(A,B):=\Map(\kkG(A),\kkG(B))$ for the mapping spectrum. Since $\kkG(\C)$ is the tensor unit of $\KKG$ we can define the
 commutative ring spectrum $R_{G}:=\KKG(\C,\C)$. {For every $A$ in $G\nCalg$ the composition $\KKG(A,\C)\times \KKG(\C,\C)\to \KKG(A,\C)$ turns $\KKG(A,\C)$ into a $R_{G}$-module.
 The functor $\KKG(-,\C)$ therefore}  naturally refines  to a functor $\KKG(-,\C)\colon {(\nCalg)^{op}}\to \Mod(R_{G})$, and we
%
%
%
%
  can define the functor
$$KU_{G}\colon G\Orb\to \Mod(R_{G})\ , \quad S\mapsto \KKG(C(S),\C)\ ,$$
 where $C(S)$ denotes the $G$-$C^{*}$-algebra of $\C$-valued functions on $S$. The $G$-action on $C(S)$ and  the action of morphisms in $G\Orb$ {are} implemented via pull-back of functions. 
 
For a subgroup $H$ of $G$ we  have an isomorphism $C(G/H)\simeq \Ind_{H}^{G}(\C)$, and hence, by~\cite[Thm.~1.23.1]{KKG}, an equivalence
$\KKG(C(G/H),\C)\simeq \KKH(\C,\C)=R_{H}$.  
This turns~$R_{H}$ into an $R_{G}$ module. On the level of $\pi_{0}$ this module structure is given by
the restriction homomorphisms $R(G)\to R(H)$. 
  \hrB
  \end{rem}


  Note that 
   the equivariant $K$-homology of a $G$-topological space $X$   in our notation (see Definition \ref{uiweogwerwwefwef} and \eqref{v4toi4hgkjnetkvevv})  can be expressed as
  $KU_{G}^{G}( X)$.
     \footnote{Note that the upper index $G$ here refers to Definition \ref{uiweogwerwwefwef}.}

   It is known that $\pi_{0}R_{G}$ is isomorphic to the representation ring $R(G)$ of $G$.
   The trace  provides a ring  homomorphism
   $\Tr\colon R(G)\to C(G)$ which   takes values in the subalgebra  of conjugation invariant functions in $G$.
   
  Let $\gamma$ be a conjugacy class.
  \begin{ddd}\label{rgiogerg43t34t43t}
  We let   $(\gamma)$  be the ideal in $R(G)$ of elements $\rho$ with $\Tr(\rho)(g)=0$ for all~$ g$ in $\gamma$. \end{ddd}
  Note that $(\gamma)$ is a prime ideal.
  
  \begin{rem}\label{wtkoghpgferfwerg}
 The ring structure of $R_{G}$ induces for  $\alpha$ in $R(G)\cong \pi_{0}(R_{G})$   a multiplication map $\alpha\colon R_{G}\to R_{G}$.
 We have a symmetric {monoidal}  (Bousfield) localization \begin{equation}\label{gergregerergregregreregerg1} (-)_{(\gamma)}:\Mod(R_{G})\to L_{(\gamma)}\Mod(R_{G})\end{equation} generated by the set of morphisms
 $(R_{G}\xrightarrow{\alpha} R_{G})_{\alpha\in R(G)\setminus (\gamma)}$. 
We   {also} have a canonical   functor \begin{equation}\label{gergregerergregregreregerg} \Mod(R_{G,(\gamma)})\stackrel{}{\to}   L_{(\gamma)}\Mod(R_{G})\end{equation}
 which is actually a symmetric monoidal equivalence (see, \emph{e.g.,} the  {text after}  \cite[Prop.~7.2.3.25]{HA}).
The symmetric monoidal morphism $loc_{(\gamma)}\colon \Mod(R_{G})\xrightarrow{loc_{(\gamma)}}  \Mod(R_{G,(\gamma)})$ is the composition of the localization~\eqref{gergregerergregregreregerg1} with the inverse of \eqref{gergregerergregregreregerg}.
The pull-back $l^{*}\colon \Mod(R_{G,(\gamma)})\to  \Mod(R_{G})$ along the   localization   morphism   $l\colon R_{G}\to R_{G,(\gamma)}$  of commutative algebras in $\Sp$
is lax-symmetric monoidal.\hrB
\end{rem}

As explained in Remark \ref{wtkoghpgferfwerg}, we have a   lax symmetric monoidal localization functor \begin{equation}\label{dewdwekjnkj23wewf}
\Mod(R_{G})\xrightarrow{loc_{(\gamma)}}  \Mod(R_{G,(\gamma)})\stackrel{l^{*}}{\to}  \Mod(R_{G})\ , \quad  M\mapsto M_{(\gamma)}\ ,
\end{equation}
   Applying the localization  functor  \eqref{dewdwekjnkj23wewf} object-wise to $KU_{G}$
we get the functor
$$KU_{G,(\gamma)}\colon G\Orb\to  \Mod(R_{G})\ .$$

\begin{lem}[\cite{segalk}]\label{wfiowefwefewfewfw}
The  functor $KU_{G,(\gamma)}$ vanishes on $\cF(\gamma)$.
\end{lem}
\begin{proof}
Let $H$ be a subgroup in $\cF(\gamma)$.
Then Segal   \cite[{Prop.~3.7}]{segal} {(see also \cite[Sec.~4]{segalk})} has shown that there exists an element $\eta$ in $R(G)$ with
\begin{enumerate}
\item 
$\eta_{|H}=0$ and \item\label{efiowf32r23r23r} $\Tr(\eta)(g)\not=0$ for  all $g$ in $\gamma$. \end{enumerate}
We have an isomorphism
  $\pi_{0}KU_{G}(G/H)\cong R(H)$ as an $R(G)$-module, where the $R(G)$-module structure on $R(H)$ is given by the restriction homomorphism $R(G)\to R(H)$, see Remark \ref{erthopetrghertge}.

By {Assumption}~\ref{efiowf32r23r23r}. multiplication  by $\eta$ 
acts as an isomorphism on the $\Z$-graded $R(G)$-module  $\pi_{*}KU_{G,(\gamma)}(G/H)$.  On the other hand, the multiplication by $\eta$ on $\pi_{*}KU_{G,(\gamma)}(G/H)\simeq R(H)_{(\gamma)}$ is the multiplication by $\eta_{|H}$  and therefore vanishes.

This implies that  $ \pi_{*} KU_{G,(\gamma)}(G/H)\cong 0$ and hence
$KU_{G,(\gamma)}(G/H)\simeq 0$.
\end{proof}

%
%
%

The Abstract Localization Theorem II~\ref{wriowwgwegewgewgewgw} can {now} be applied.

Assume that $G$ is a finite group, $\gamma$ is a conjugacy class in $G$, and that $X$ is a $G$-topological space. 
\begin{theorem}[Atiyah-Segal Localization Theorem]\label{wriowwgwegewgewgewgwas}
If $X^{\gamma}$  is closed in $X$  and admits a 
open invariant neighbourhood $N$ such that $X^{\gamma}\to N$ is an equivariant weak equivalence, then   the inclusion
$X^{\gamma}\to X$ induces an equivalence
$$KU^{G}_{G,(\gamma)}( X^{\gamma} ) \to KU^{G}_{G,(\gamma)}( X) \ .$$
\end{theorem}

In order to compare this with the classical Atiyah-Segal localization theorem \cite[Prop.~4.1]{segalk} note that, since localization commutes with colimits, we have an equivalence
$$KU^{G}_{G,(\gamma)}( X)\simeq 
KU_{G}^{G} ( X)_{(\gamma)}\ ,$$
i.e., we can interchange the order of localizing and forming homology. 

    \section{Fixed points in bornological coarse spaces}

\subsection{Conventions}
 {In what follows we assume that the reader is familiar with the notions of bornological coarse spaces and coarse homology theories {as developed in} \cite[{Part I}]{buen} {and} \cite[{Sec.~2,3 \& 4}]{equicoarse}.}

We let $\BC$ and $G\BC$ be the categories of bornological coarse spaces and $G$-bornological coarse spaces. By 
 \begin{equation}\label{344frf}
\Yo:\BC\to \Spc\cX
\end{equation}
and \begin{equation}\label{34g453oighj43oig43lrf334frf}
\Yo_{G}:G\BC\to G\Spc\cX
\end{equation}
 we denote the functors which send a bornological coarse space or $G$-bornological coarse space to its {motivic coarse}  space \cite[{Def.~3.34}]{buen}, \cite[Sec.~4.1]{equicoarse}. The stable versions of these functors are denoted by 
$$\Yo^{s}:\BC\to \Sp\cX$$ and $$\Yo_{G}^{s}:G\BC\to G\Sp\cX\ .$$

Let $\bC$ be a cocomplete stable $\infty$-category. A $\bC$-valued coarse homology theory
is a functor $$E:\BC\to \bC$$ satisfying the conditions: coarse invariance, excision, vanishing on flasques, and $u$-continuity \cite[Def.~4.{22}]{buen}.  
Similarly, a $\bC$-valued equivariant coarse homology theory is a functor $$E:G\BC\to \bC$$ satisfying the equivariant versions of these properties \cite[Def.~3.10]{equicoarse}.
Occasionally we will consider an additional property called continuity, 
\cite[Def.~5.15]{equicoarse}.

By the universal property of the functors $\Yo$ and $\Yo_{G}$, pull-back along these functors {induces} equivalences between  the $\infty$-categories of $\bC$-valued coarse homology theories or 
$\bC$-valued equivariant coarse homology theories  and the functor categories
$\Fun^{\colim}(\Spc\cX,\bC)$ or $\Fun^{\colim}(G\Spc\cX,\bC)$, respectively (see e.g., \cite[Cor.~{3.37}]{buen}).

If $X$ is in $\BC$ (or $G\BC$) and $E$ belongs to $\Fun^{\colim}(\Spc\cX,\bC)$ (or $\Fun^{\colim}(G\Spc\cX,\bC)$), then we will usually write
$E(X)$ instead of $E(\Yo(X))$ (or $E(\Yo_{G}(X))$).

 By the universal property of the stabilization maps
$\Spc\cX\to \Sp\cX$ and $G\Spc\cX\to G\Sp\cX$ and stability of $\bC$ we have further equivalences,
 \begin{equation}\label{ewflkwenfjkewneiwjffefff2}\Fun^{\colim}(\Spc\cX,\bC)\simeq \Fun^{\colim}(\Sp\cX,\bC)  \end{equation}
and  \begin{equation}\label{ewflkwenfjkewneiwjffefff21} \Fun^{\colim}(G\Spc\cX,\bC)\simeq \Fun^{\colim}(G\Sp\cX,\bC)\ . 
\end{equation}

If $S$ is a $G$-set, then we let $S_{min,max}$ denote the $G$-{bornological} coarse space $S$ with the minimal coarse structure and the maximal bornology {\cite[Ex.~2.5]{equicoarse}}. The group $G$  itself has a  canonical $G$-coarse structure and gives rise to the $G$-bornological coarse space $G_{can,min}$ \cite[Ex.~2.4]{equicoarse}.

 The categories $\BC$ and $G\BC$ have symmetric monoidal structures denoted by $\otimes$ \cite[Ex.~{2.32}]{buen},  \cite[Ex.~2.17]{equicoarse}.
Furthermore, the categories $\Spc\cX$ and $G\Spc\cX$ have symmetric monoidal structures, also denoted by $\otimes$, which are essentially uniquely determined by the requirement that the functors $\Yo$ and $\Yo_{G}$ refine to symmetric monoidal functors. 

For $X \in G\Spc\cX$  and $Y\in G\BC$ we will  write
$Y\otimes X$ instead of $\Yo_{G}(Y)\otimes X$.

 \subsection{Motivic fixed points}

  Let $G$ be a group and $H$ be a subgroup.  By $$W_{G}(H):=N_{G}(H)/H$$ we denote the Weyl group of $H$ in $G$.
  
  If $X$ is a $G$-bornological coarse space, then we can consider the set of $H$-fixed points~$X^{H}$ in the underlying $G$-set of $X$. In this section we equip $X^{H}$ with the structure of a $W_{G}(H)$-bornological coarse space.
 Our main result is that the functor
 $$X\mapsto \Yo_{W_{G}(H)}(X^{H})\colon G\BC\to W_{G}(H)\Spc\cX$$
has the property of an (unstable) equivariant coarse homology theory. Therefore, it essentially uniquely factorizes over a motivic fixed point functor
 $$I^{H}\colon G\Spc\cX\to W_{G}(H)\Spc\cX\ .$$

 We start  
 with a  $G$-bornological coarse space $X$.  Then the  subset $X^{H}$ of $H$-fixed points 
 of~$X$  has   an induced action by the Weyl group $W_{G}(H)$. The embedding of  the set $X^{H}$ into  the completion $B_{G}X$\footnote{The completion $B_{G}X$ is the $G$-bornological coarse space obtained from $X$ by replacing its bornology $\cB$ by the   bornology $G\cB$ generated by the subsets $GB$ for all $B$ in $\cB$. The choice of this convention is motivated by the proof of Lemma \ref{geioogergregregregrege} below. } of $X$  induces a bornological coarse structure on $X^{H}$, and $X^{H}$ with this structure is a  $W_{G}(H)$-bornological coarse  space. 
   
  If $f\colon X\to Y$ is a morphism of $G$-bornological coarse spaces, then it   restricts to a morphism of $W_{G}(H)$-bornological coarse spaces $f^{H}\colon X^{H}\to Y^{H}$.
  We therefore get a functor
 $$(-)^{H}\colon G\BC\to W_{G}(H)\BC\ , \quad  X\mapsto X^{H}\ .$$ 

 
 \begin{lem}\label{f4iufof34f234f3f42}The composition
   \begin{equation}\label{}
{G}\BC\xrightarrow{(-)^{H}}W_{{G}}(H)\BC\xrightarrow{\Yo_{W_{{G}}(H)}} W_{{G}}(H)\Spc\cX
\end{equation}
   satisfies:
   \begin{enumerate}
   \item coarse invariance
   \item excision
   \item vanishing on flasques
   \item $u$-continuity.
   \end{enumerate}
\end{lem}
\begin{proof}
We use that  the functor $\Yo_{W_{G}(H)}$ (see Eq.~\eqref{34g453oighj43oig43lrf334frf}) has these properties.
   \begin{enumerate}
   \item coarse invariance: In $W_{G}(H)\BC$ the morphism  $(\{0,1\}\otimes X)^{H}\to X^{H}$ is isomorphic to the morphism 
   $ \{0,1\}\otimes X^{H}\to X^{H}$, and $\Yo_{W_{G}(H)}$ sends the latter to an equivalence.
   \item excision: If $(Z,\cY)$ is a complementary pair on $X$, then
   $(Z^{H},\cY^{H})$ is a complementary pair on $X^{H}$, where $\cY^{H}$ is defined by  member-wise application of $(-)^{H}$.  If we apply $(-)^{H}$ to  the square
   $$\xymatrix{Z\cap \cY\ar[r]\ar[d]&\cY\ar[d]\\Z\ar[r]&X}\ ,$$
   then we get the square 
    $$\xymatrix{Z^{H}\cap \cY^{H}\ar[r]\ar[d]&\cY^{H}\ar[d]\\Z^{H}\ar[r]&X^{H}}\ .$$
       The functor $\Yo_{W_{G}(H)}$ sends this square to a push-out square.
      \item  flasques: If $f\colon X\to X$ implements flasqueness of $X$, then $f^{H}\colon X^{H}\to X^{H}$  implements flasqueness of $X^{H}$. Here it is important to know that  the completion functor $X\mapsto B_{G}X$ preserves flasqueness.
            It follows that $\Yo_{W_{G}(H)}(\emptyset)\to \Yo_{W_{G}(H)}(X^{H})$ is an equivalence.
     \item $u$-continuity:
     Let $X$ be a $G$-bornological coarse space. Then we have 
     \begin{eqnarray*}
     \colim_{U\in \cC^{G}_{X}} \Yo_{W_{G}(H)}((X_{U})^{H})&\stackrel{!}{\simeq}& \colim_{U\in \cC^{G}_{X}} \colim_{V\in \cC^{G}_{X_{U}}}\Yo_{W_{G}(H)}((X^{H})_{V\cap (X^{H}\times X^{H})})\\&\stackrel{!!}{\simeq}&
     \colim_{U\in \cC^{G}_{X}}  \Yo_{W_{G}(H)}((X^{H})_{U\cap (X^{H}\times X^{H})})\\&\stackrel{!}{\simeq}&
     \Yo_{W_{G}(H)}(X^{H})\ ,
\end{eqnarray*}    
        \end{enumerate}
        where we use $u$-continuity of $\Yo_{W_{G}(H)}$ at the equivalences marked by $!$ and a cofinality consideration at $!!$.
  \end{proof}
  
  
%
%

 By the universal property of the functor $\Yo_{G}\colon G\BC\to G\Spc\cX$ the composition
 $\Yo_{W_{G}(H)}\circ (-)^{H}$ has an essentially unique colimit-preserving factorization $I^{H}$ through~$\Yo_{G}$:
   $$\xymatrix{G\BC \ar[dr]^{ \Yo_{W_{G}(H)}\circ (-)^{H}}\ar[d]^{\Yo_{G}}& \\G\Spc\cX\ar[r]^{I^{H}}&W_{G}(H)\Spc\cX}\ .$$
     
   \begin{ddd}\label{ghioegfergrereg34t}
   We call $I^{H}:G\Spc\cX \to W_{G}(H)\Spc\cX$ the motivic $H$-fixed point functor.
   \end{ddd}
   
   So if $X$ is in $G\Spc\cX$, then we can talk about its $H$-fixed points $I^{H}(X)$ as an object of $W_{G}(H)\Spc\cX$.
   
   \begin{ex} The formation of fixed points is compatible with the cone functor~\cite[Sec.~9.4]{equicoarse} $$\cO \colon G\UBC\to  G\BC$$ 
   from $G$-uniform bornological coarse spaces to $G$-bornological coarse spaces.
  Let $A$ be a $G$-uniform bornological coarse space. Then the set of $H$-fixed points $A^{H}$ has 
   a $W_{G}(H)$-uniform bornological coarse structure{,} where the bornology is again generated by the {subsets}  $GB\cap A^{H}$ for all bounded subsets $B$ of $A$.

%
%

%

   \begin{lem}\label{gireogoerg34t34erg}
   We have an equivalence $  \cO (A)^{H}\simeq \cO (A^{H})$ in $W_{G}(H)\BC$.
   \end{lem}
\begin{proof}
We have $\cO(A):=([0,\infty)\otimes A)_{h}$, were $h$ stands for the hybrid structure associated to the family
of subsets $([0,r]\times A)_{r\in \R}$ and the uniform structure \cite[Sec.~9.1 \& 9.2]{equicoarse}. 

 We have a $W_{G}(H)$-equivariant bijection of sets
$([0,\infty)\times A)^{H}\cong [0,\infty)\times A^{H}$. So we must check that it is an isomorphism of $W_{G}(H)$-bornological coarse spaces.
The bornology on the right-hand side is generated by the sets
$L\times (GB\cap A^{H})$ for bounded subsets $L$ of $[0,\infty)$ and $B$ of $A$.
We have a bijection
$L\times (GB\cap A^{H})=(L\times GB)\cap ([0,\infty)\times A)^{H}$. This shows that these sets also generate the bornology of $([0,\infty)\times A)^{H}$.

 Let $U$ be an invariant entourage of $([0,\infty)\otimes A)_{h}$.
 Then we get an entourage $U^{\prime}$ of $([0,\infty)\otimes A)_{h}^{H}$ by restriction.
It corresponds to an entourage of $([0,\infty)\otimes A^{H})_{h}$, and these entourages generate the hybrid structure.
We get an isomorphism of $W_{G}(H)$-bornological coarse spaces
$$ ([0,\infty)\otimes A)_{h}^{H} \cong ([0,\infty)\otimes A^{H})_{h} \ . \qedhere$$
 \end{proof} \heB
    \end{ex}

\subsection{The  motivic orbit functor}\label{greiojgioergergergergerg}

In this section we refine the construction of the motivic fixed point functor Definition \ref{ghioegfergrereg34t} so that we take its dependence on the subgroup $H$ into account properly.

 Recall that  $G\Orb$ denotes the { full subcategory of {the category} $G\Set$ of transitive $G$-sets and equivariant maps}.
 Let $X$ be a $G$-set.
For every $S$ in $G\Orb$ we consider the set 
\begin{equation}\label{v4toihfkj3rof3rf3rfref}
X^{(S)}:=\Hom_{G\Set}(S,X)\ .
\end{equation} 
  
 For every point  $s$ in $S$ the evaluation $e_{s}\colon X^{(S)}\to X$  provides a bijection between the set~$X^{(S)}$ and the subset $X^{G_{s}}$ of $X$, where $G_{s}$ denotes the stabilizer of $s$ in $G$.
 
 We now assume that $X$ is a $G$-bornological coarse space.
We equip $X^{(S)}$ with the   {bornological} coarse structure induced by $e_{s}\colon X^{(S)}\to B_{G}X$. This structure does not depend on the choice of the point $s$ in $S$. From now on $X^{(S)}$ denotes the bornological coarse space just described.

We consider a morphism $\phi\colon S\to T$ in $G\Orb$. It induces a map
$\phi^{*}\colon \Hom_{G\Set}(T,X)\to \Hom_{G\Set}(S,X)$, i.e., a map between the underlying sets of $ X^{(T)}$  and $X^{(S)}$.

\begin{lem}\label{wefoiwefwefwef}
The map $\phi^{*}:X^{(T)}\to X^{(S)}$ is a morphism of  bornological coarse spaces.
\end{lem}
\begin{proof}
We fix a point $s$ in $S$ and let $t:=\phi(s)$. 
Note that $G_{s}$ is a subgroup of $G_{t}$. We get the diagram of {sets}
$$\xymatrix{X^{(T)}\ar[r]^{e_{t}}\ar[d]^{\phi^{*}}&X^{G_{t}}\ar[d]\ar[r]&B_{G}X\ar@{=}[d]\\X^{(S)}\ar[r]^{e_{s}}&X^{G_{s}}\ar[r]&B_{G}X}\ .$$
Since the subsets $X^{G_{t}}$ and $X^{G_{s}}$ have the induced structures from $B_{G}X$ the middle vertical
map is a morphism of bornological coarse spaces. Since $e_{t}$ and $e_{s}$ are isomorphisms of bornological coarse spaces by definition, also $\phi^{*}$ is a morphism of bornological coarse spaces. 
\end{proof}

  For every $G$-bornological coarse space $X$ we have just  
 defined a functor
$$\hat Y(X)\colon G\Orb^{op}\to \BC\ , 
{\quad S\mapsto X^{(S)}}
$$
This construction is also  functorial in $X$ and gives a functor
$$\hat Y\colon G\BC\to \Fun(G\Orb^{op},\BC)\ .$$
Let $\Yo\colon \BC\to {\Spc}\cX$   be the functor in Eq~\eqref{344frf}  which sends a bornological coarse space to its motive.
\begin{lem}\label{34i34o2t23r23r23r} The composition \begin{equation}\label{f34fpo4fr3fg3f}
\tilde Y\colon G\BC\xrightarrow{\hat Y} \Fun(G\Orb^{op},\BC)\xrightarrow{\Yo}  \Fun(G\Orb^{op},\Spc\cX)
\end{equation}
satisfies
\begin{enumerate}
\item coarse invariance
\item excision
\item vanishing on flasques
\item $u$-continuity.
\end{enumerate}
\end{lem}
\begin{proof}
The proof is completely analogous to the proof of Lemma \ref{f4iufof34f234f3f42}.
\end{proof}

\begin{rem}
We use the same notation $\tilde Y$ as in \eqref{v4toi4hgkjnetkvevv} since the functor \eqref{f34fpo4fr3fg3f} just described   is  the bornological coarse analogue of the functor there.\hrB
\end{rem}

%
By the universal property of the functor $\Yo_{G}\colon G\BC\to G\Spc\cX$ the functor $\tilde Y$ has  an essentially unique colimit-preserving factorization $Y$
  $$\xymatrix{G\BC \ar[dr]^{  \tilde  Y}\ar[d]^{\Yo_{G}}& \\ G\Spc\cX\ar[r]^{Y}& \Fun(G\Orb^{op},\Spc\cX)}$$
 through  $\Yo_{G}$   \begin{ddd}\label{rgiorgefewfwefwef}
We call $Y\colon G\Spc\cX\to \Fun(G\Orb^{op},\Spc\cX)$ the motivic orbit functor. \hB
\end{ddd}

\begin{rem}\label{efiewfwefwefw}
By Elmendorf's theorem the topological analogue of $Y$ is  the equivalence~\eqref{rvelkm3l4kgferveve}.
%
%
%
In the case of bornological coarse spaces
we do not know an analogue of Elmendorf's theorem,  i.e., we do not know whether $Y$ is an equivalence. \hrB
\end{rem}

\begin{rem}
The categories $G\Spc\cX$, $\Spc\cX$ and $ \Fun(G\Orb^{op},\Spc\cX)$ are presentable. By construction the functor
$$Y\colon G\Spc\cX\to \Fun(G\Orb^{op},\Spc\cX)$$ preserves all colimits. It is therefore a part of an adjunction
$$Y:G\Spc\cX\leftrightarrows \Fun(G\Orb^{op},\Spc\cX):Z\ .$$ \hrB
\end{rem}

 
\subsection{Bredon-style  equivariant coarse homology theories} 
 
 In this section we introduce a construction of equivariant coarse homology theories from diagrams of (non-equivariant) coarse homology theories indexed by the orbit category. These equivariant  coarse homology theories
 are called  Bredon-style  equivariant coarse homology theories (see Definition \ref{brstyle}).

Let $\bC$ be a cocomplete stable $\infty$-category and 
 consider a functor 
 $$E\colon G\Orb\to \Fun^{\colim}(\Spc\cX,\bC)\ .$$
 If $X$  is in $\Fun(G\Orb^{op},\Spc\cX)$, then we can form the functor
 $$E\circ X\colon G\Orb\times G\Orb^{op}\to \bC\ , \quad (S,T)\mapsto E(S)(X(T))\ .$$

  \begin{ddd}\label{fuifehwieufhf23f2323rr233r2r3}
 We define $$E^{G}(X):= \int^{G\Orb} E\circ X\ .$$
 \end{ddd}
Note that the construction is functorial in $E$ and $X$. 

Recall the motivic orbit functor $Y\colon G\Spc\cX\to \Fun(G\Orb^{op},\Spc\cX)$ defined in Definition~\ref{rgiorgefewfwefwef}.
We consider a functor  $E\colon G\Orb\to \Fun^{\colim}(\Spc\cX,\bC)$.
\begin{lem}\label{ewfiwefwefewf}
The functor $$E^{G}\circ Y\colon G\Spc\cX\to \bC$$
preserves colimits.
\end{lem}
\begin{proof}
We  note that $Y$ preserves colimits by construction. Since $E(S)$ preserves colimits for every $S$ in $G\Orb$ the functor 
$X\mapsto E\circ Y(X)$ preserves colimits. Finally, forming  the coend preserves colimits.
\end{proof}
Lemma \ref{ewfiwefwefewf} has the following immediate corollary:
\begin{kor} \label{gih42iuhjfierfwfef}
The functor $$E^{G}\circ \tilde Y\simeq E^{G}\circ Y\circ \Yo_{G}\colon G\BC\to \bC$$
is a $\bC$-valued equivariant  coarse   homology theory.
\end{kor}

Usually we will shorten the notation and write
$E^{G}(X)$ instead of $E^{G}(Y(\Yo_{G}(X)))$ for $X$ in $G\BC$.

%
\begin{rem}\label{jkgyglglygy}
Since we do not know  and  do not expect that $$Y\colon G\Spc\cX\to \Fun(G\Orb^{op},\Spc\cX)$$ is an equivalence  (see Remark \ref{efiewfwefwefw}),  in contrast to the topological case
we do not expect     that all $\bC$-valued equivariant coarse homology theories
are of the form $E^{G}$ for some functor $$E\colon G\Orb\to \Fun^{\colim}(\Spc\cX,\bC)\ .$$
 \hrB
\end{rem}


\begin{ddd}\label{brstyle}
 We call the  equivariant coarse homology theories of the form $E^{G}$  from Definition \ref{fuifehwieufhf23f2323rr233r2r3} for some $E\colon G\Orb\to \Fun^{\colim}(\Spc\cX,\bC)$     Bredon-style equivariant coarse homology theories. 
 \end{ddd}
 \begin{ex}
 Let $F\colon \Spc\cX\to \bC$ be a colimit-preserving functor, i.e.~{a} non-equivariant coarse homology theory.
 Then we can consider the constant functor
 $$\underline{\underline{F}}\colon G\Orb\to \Fun^{\colim}(\Spc\cX,\bC)$$
 with value $F$.
 For $X$ in $\Fun(G\Orb^{op},\Spc\cX)$ the
functor $\underline{\underline{F}}\circ X\colon G\Orb\times G\Orb^{op}\to \bC$ is constant in the first argument. 
This implies
$$ \underline{\underline{F}}^{G}(X)\simeq \int^{G\Orb} \underline{\underline{F}}\circ X\simeq \colim_{G\Orb^{op}} (F\circ X) \simeq  F(\colim_{G\Orb^{op}} X )\ ,$$
where for the second equivalence we use that $F$ preserves colimits.
\heB
\end{ex}

\begin{ex}
We consider a  colimit-preserving functor
  $F\colon \Spc\cX\to \bC$ and  form the functor
$$\underline{F}\colon G\Orb\to \Fun^{\colim}(\Spc\cX,\bC)\ , \quad \underline{F}(S)(X):=\Res^{G}(S)\otimes F( X)\ ,$$
where $\Res^{G}(S)$ is the underlying set of $S$.
%

For  $X$ in $\Fun(G\Orb^{op},\Spc\cX)$
we then  have  a canonical equivalence \begin{equation}\label{qwefwfewqd}\underline{F}^{G}(X)\simeq F(X(G))\ .
\end{equation} 
In order to derive this equivalence
%
we note that the  functor $S\mapsto \Res^{G}(S)$ is corepresented by the object $G$, i.e., we have an equivalence
$ \Hom_{G\Orb} (G,S)\simeq \Res^{G}(S)$.  
 It follows from the co-Yoneda lemma  that $$
 \underline{F}^{G}(X) \simeq \int^{G\Orb} \Hom_{G\Orb} (G,-)\otimes F(X(-)) 
  \simeq F(X(G))\ .
 $$
  \heB
  \end{ex}

\begin{ex} 
Let $\rho\colon H\to G$ be a  homomorphism of groups. 
It induces an induction functor  $$\Ind_{\rho}\colon H\Orb\to G\Orb\ , \quad S\mapsto G\times_{H}S\ .$$ 
We write $\rho^{*}$ for the pull-back of functors along $\Ind_{\rho}$. Since $\Fun (\Spc\cX,\bC)$ is cocomplete,   we have an adjunction
$$\rho_{!}: \Fun(H\Orb,\Fun (\Spc\cX,\bC))\leftrightarrows  \Fun(G\Orb,\Fun (\Spc\cX,\bC)):\rho^{*}\ .$$
We  further note that
$\rho_{!}$ restricts to a functor $$\rho_{!}\colon \Fun(H\Orb,\Fun^{\colim} (\Spc\cX,\bC))\to  \Fun(G\Orb,\Fun^{\colim} (\Spc\cX,\bC))\ .$$

For a complete $\infty$-category $\bD$ (we will apply this for $D=\Spc$) we also have an adjunction 
$$\rho^{*}:\Fun(H\Orb^{op},\bD)\leftrightarrows \Fun(G\Orb^{op},\bD):\rho_{*}\ .$$
For $X$ in $\Fun(G\Orb^{op},\Spc\cX)$ and $F$ in $\Fun(H\Orb,\Fun^{\colim} (\Spc\cX,\bC))$   we have 
a canonical induction-restriction 
 equivalence
\begin{equation}\label{}(\rho_{!}F){^G}(X)\simeq F^{H}(\rho^{*}X)\ .
\end{equation}  
Its derivation involves unfolding definitions and formal properties of adjunctions involving various right- and left Kan extension functors.
\heB  \end{ex}

 \subsection{Associated Bredon-style homology theories}\label{f42oifj2o3f32f32e}
 
In the present section we  associate to every equivariant coarse homology theory $E$ {an} approximation  $E^{\Bredon}\to E$ by a Bredon-style equivariant coarse homology theory. Then we investigate the question under which conditions on $E$ and $X$ {in $G\BC$}
the  morphism
$$E^{\Bredon}(X)\to E(X)$$ 
is an equivalence, {and we will provide examples where it is not.}


We consider the composition of functors
\begin{equation}\label{fwefewfewf}
   G\Orb \times \BC  \xrightarrow{(-)_{min,max}\otimes-} G\BC \xrightarrow{\Yo_{G}} G\Spc\cX   
\end{equation}  which sends 
$ (S,X)$  in $  G\Orb \times \BC$ to $ \Yo_{G}(S_{min,max}\otimes X)$ in $G\Spc\cX$.
Using~\cite[Lem.~4.17]{equicoarse}, we obtain an essentially  unique factorization
$$\xymatrix{G\Orb \times  \BC\ar[d]_{\id\times \Yo}\ar[r]^(0.65){\eqref{fwefewfewf}}&G\Spc\cX \\G\Orb \times \Spc\cX\ar@{-->}[ru]_{\hspace{0.9cm}(S,X)\mapsto S_{min,max}\otimes X }&}$$
 which preserves colimits  in its second argument.

  Let $\bC$ be a cocomplete stable $\infty$-category and
 consider   $E$ in $\Fun^{\colim}(G\Spc \cX,\bC)$. 

  \begin{ddd} \label{etgioegggergergergeg}
 We define the functor \begin{equation}\label{g54ljklefwefewfwf}
\underline{E}\colon G\Orb\to \Fun^{\colim}(\Spc\cX,\bC)\ , \quad \underline{E}(S)(X):=  E( S_{min,max}\otimes X)\ .
\end{equation}
\end{ddd}
Furthermore 
we let $\underline{E}^{G}$ in $\Fun({\Fun(G\Orb^{op},\Spc\cX)},\bC)$ denote the result of the application of 
  Definition \ref{fuifehwieufhf23f2323rr233r2r3} to $\underline{E} $.

  Recall the Definition \ref{rgiorgefewfwefwef} of the motivic orbit functor $Y {\colon G\Spc\cX\to \Fun(G\Orb^{op},\Spc\cX)}$.

\begin{ddd}\label{rgfrihgioffwfwfwefewf}
We call $E^{\Bredon}:=\underline{E}^{G}\circ Y$ the Bredon-style equivariant coarse homology theory associated to $E$.
\end{ddd}


The name \emph{equivariant coarse homology theory} is justified by Corollary   \ref{gih42iuhjfierfwfef}.

 \begin{rem}
 The functor $$G\Orb\to \Fun^{\colim}(\Spc\cX,G\Spc\cX)\ , \quad S\mapsto (X\mapsto S_{min,max}\otimes X)$$
 is an attempt to produce an analogue of the Yoneda embedding $$i\colon G\Orb\to G\Top[W_{G}^{-1}]  \simeq 
 \Fun^{\colim}( \Spc,  G\Top[W_{G}^{-1}]) \ , \quad S\mapsto S_{disc}\mapsto (X\mapsto S_{disc}\otimes X)$$ in equivariant homotopy theory.  By Elmendorf's theorem the image of the latter   generates  its target $G\Top[W_{G}^{-1}]$ which implies that  every equivariant homology theory $F$   (a colimit preserving functor  $F\colon G\Top[W_{G}^{-1}]\simeq \PSh(G\Orb)\to \bC$) is a Bredon-style homology theory of the form
 $F\simeq E^{G}$ as  introduced in Definition \ref{uiweogwerwwefwef} with $E:=i^{*}F$.
 
 As we will see below we do not have analogous statements in the coarse case.
 By Corollary~\ref{wtgijowgwrefwerfrfw}, the image of
 $G\Orb \times \Spc\cX\to G\Spc\cX$ does not generate the target. Furthermore,
 by Corollary \ref{weijgowegrefweferf}  we know that in general the Bredon  approximation \eqref{f24oif} 
 of an equivariant coarse homology theory is not an equivalence. \hrB
 \end{rem}

 Recall from  \eqref{v4toihfkj3rof3rf3rfref} the notation  $X^{(S)}$   for the bornological coarse space given by the set $\Hom_{G\Set}(S,X)$  with the structures described in Subsection \ref{greiojgioergergergergerg}.
 For $X$ in $G\BC$ and $S$ in $G\Orb$ we consider the evaluation map \begin{equation}\label{qeklnklnqklnwefqefqwefqefq}
e_{S,X}\colon S_{min,max}\otimes X^{(S)}\to X\ .
\end{equation}
 \begin{lem}\label{geioogergregregregrege}
 The map $e_{S,X}$ is a morphism of $G$-bornological coarse spaces.
 \end{lem}
 \begin{proof}
 The map $e_{S,X}$ is $G$-equivariant. 
 Let $U$ be a $G$-invariant entourage of $X$ and $U_{S}$ be the induced entourage on $X^{(S)}$ via the evaluation $e_{s_{0}}\colon X^{(S)}\to X$ at
 $s_{0}$ in $S$. Then $$e_{S,X}(\diag_{S}\times U_{S})=\bigcup_{g\in G} g e_{s_{0}}(U_{S})\ .$$ Since $U$ was invariant and
 $e_{s_{0}}(U_{S})\subseteq U$ we have
 $e_{S,X}(\diag_{S}\times U_{S})\subseteq U$. This shows that~$e_{S,X}$ is a controlled map.
 
 Let $B$ be a bounded subset of $X$. Then $$e_{S,X}^{-1}(B)=\bigcup_{s\in S} \{s\}\times e_{s}^{-1}(B)\ .$$
 Now note that $e_{s}^{-1}(B)\subseteq e_{s_{0}}^{-1}(GB)$ for every $s$ in $S$, and this subset is bounded in $X^{(S)}$ by definition of the bornology. It is at this point where we employ the definition of the bornology of $X^{(S)}$ as induced by the $G$-completion of $X$.
Hence we have
$$e_{S,X}^{-1}(B)\subseteq S\times e_{s_{0}}^{-1}(GB)\ .$$
Consequently, 
$e_{S,X}^{-1}(B)$ is bounded since $S$ is bounded in $S_{min,max}$.
 \end{proof}

For {a} fixed $S$ in $G\Orb$ the family of  maps  $(e_{S,X})_{X\in G\BC}$ is a transformation of functors
$$e_{S}:S_{min,max}\otimes (-)^{(S)} \to \id_{G\BC}:G\BC\to G\BC\ ,$$
{and it} extends to motives   $$e_{S}:S_{min,max}\otimes (-)^{(S)} \to \id_{G\Spc\cX}:G\Spc\cX\to G\Spc\cX\ .$$

Let $X$ be in $G\Spc\cX$ and note that $X^{(S)}\simeq Y(X)(S)$ by definition Definition \ref{rgiorgefewfwefwef} of $Y$.
 The collection of morphisms $(e_{S})_{S\in G\Orb}$  {induces,} by the universal property of the  coend, 
a  transformation
$$e^{G}\colon\int^{S\in G\Orb} S_{min,max}\otimes Y(-)(S) \to \id_{G\Spc\cX}(-)$$
of endofunctors of $G\Spc\cX$.
Since the functors $E$ and $Y$ preserve  colimits we get an induced transformation  \begin{equation}\label{f24oif}
\beta_{E}:=E(e^{G}) \colon E^{Bredon}\to E
\end{equation}
which will be called the comparison map.
{We} study the following problem. 
 
\begin{prob}
Under which conditions on $X$ or $E$  {is} the  morphism $\beta_{E,X} \colon E^{Bredon}(X)\to E(X)$
 an equivalence?
\end{prob}
This is already interesting in the universal case where $E$ is the stabilization map $G\Spc\cX\to G\Sp\cX$.

 Let $\bC$ be a cocomplete stable $\infty$-category and $E\colon G\BC\to \bC$  be an  equivariant coarse homology theory.
Let $X$ be a $G$-bornological coarse space and $W$ be a set. Then the family
$(i_{w})_{w\in W}$ of inclusions
$i_{w}\colon X\to W_{min,max}\otimes X$, $ x\mapsto (w,x)$,  induces a morphism  \begin{equation}\label{f2oi23f23f}
i_{W}\colon W\otimes E(X)\to E(W_{min,max}\otimes X)\ .
\end{equation}
 Note that $\otimes$ on the left-hand side is the tensor structure of $\bC$ over spaces (and hence sets), and $\otimes$ on the right-hand side is the tensor product in $G\BC$.

\begin{ddd}\label{egfewgewwfwefewf}
We call $E$ hyperexcisive  if for every set $W$ and every  bounded 
 $X$ in $G\BC$   the morphism \eqref{f2oi23f23f}
 is an equivalence.
 \end{ddd}
 
 \begin{rem}
 If $W$ is finite, then excisiveness of $E$ implies that   \eqref{f2oi23f23f}   is an equivalence without any assumption on $X$.    So hyperexcisiveness
 is a strengthening of excisiveness for certain infinite disjoint decompositions. \hrB
  \end{rem}
  
  \begin{ddd}
We let $G\Sp\cX_{bd}$ denote the  full {stable} subcategory of $G\Sp\cX$ generated under colimits by the objects $\Yo^{s}_{G}(X)$ for bounded $G$-bornological coarse spaces $X$.
\end{ddd}
  
 Since the functor $W\otimes -\colon\bC\to \bC$ {preserves} colimits and $E$ can be considered as an object of $\Fun^{\colim}(G\Spc\cX,\bC)$ 
 we get:
 \begin{kor}\label{rgvioerggegegerf}
 $E$ is hyperexcisive if and only of   for every set $W$ and $X$ in $G\Sp\cX_{bd}$ the natural morphism
 $W\otimes E(X)\to E(W_{min,max}\otimes X)$ is an equivalence.
 \end{kor}
 Note that in the statement above we use the equivalence \eqref{ewflkwenfjkewneiwjffefff21} in order to evaluate $E$ on a equivariant motivic coarse  spectrum in $G\Sp\cX$.

 Here is our typical example of a hyperexcisive equivariant coarse homology theory.
 Assume that $$F\colon G\BC\to \bC$$ is an equivariant coarse homology theory and form the twist
 $$E:=F_{G_{can,min}} \ , \quad E(X):=F(X\otimes G_{can,min})\ ,$$  {where $G_{can,min}$ is the canonical bornological coarse space associated to the group~$G$ -- see~\cite[Ex.~2.4]{equicoarse}}.  
 {Recall the definition  \cite[Def.~5.15]{equicoarse} of the property of continuity of a functor
 $ G\BC\to \bC$.}
 \begin{lem} \label{fewfpowef23r2} If $F$ is continuous, then
 $E $ is hyperexcisive.
  \end{lem}
\begin{proof} Let $X$ be  a bounded $G$-bornological coarse space  and 
assume that $L$ is an invariant locally finite subset of $W_{min,max}\otimes X\otimes G_{can,min}$.  We claim that the image of the projection $L\to W$ is finite. 
We have an equality  
$$L=G (L\cap (W\times X\times \{e\}))\ .$$
Hence the  image of $L\to W$ is the same as the image of $L\cap  (W\times X\times \{e\})\to W$.
Since $W\times X\times \{e\}$ is a bounded subset of $W_{min,max}\times X\times G_{can,min}$ this intersection is finite, and so is its image in $W$.

Let $\cFin(W)$ be the poset of finite subsets of $W$. Then, {by the previous paragraph, the family}
$(R\times X\times G)_{R\in \cFin(W)}$ is a trapping exhaustion (\cite[Def.~5.6]{equicoarse}) of $W_{min,max}\otimes X\otimes G_{can,min}$.
We therefore have an equivalence
\begin{eqnarray*}
E(W_{min,max}\otimes X)&\simeq& F(W_{min,max}\otimes X\otimes G_{can,min})\\&\stackrel{!!}{\simeq}&
  \colim_{R\in \cFin(W)}F(R_{min,max}\otimes X\otimes G_{can,min})\\& 
\stackrel{!}{\simeq} &\colim_{R\in \cFin(W)}R\otimes F( X\otimes G_{can,min})\\&\simeq&
W\otimes E(X)\ ,
\end{eqnarray*}
where we use excision at the marked equivalence and continuity of $F$ at $!!$.
\end{proof}

We can generalize the argument for Lemma \ref{fewfpowef23r2} to the case where  the bornological coarse space $W_{min,max}$ for a set $W$ is replaced  by an arbitrary bounded bornological coarse space~$Z$. 
Let $\pi_{0}(Z)$ {denote} the set of coarse components of $Z$ (see \cite[Def.~{2.30}]{buen}). 
 Recall that  $E=F_{G_{can,min}}$.
\begin{lem}
Assume:
\begin{enumerate}
\item $F$ is continuous. 
\item $X$ belongs to $G\Sp\cX_{bd}$.
\item  $Z$ is a bounded  bornological coarse space. 
\end{enumerate}
Then we have a natural equivalence 
\begin{equation}\label{fv4oijioejioccowicjwecwcweccwecwecwec}
E(Z\otimes X)\simeq \pi_{0}(Z)\otimes E(X)\ .
\end{equation}
 \end{lem}
 \begin{proof}
 The argument is similar as for   Lemma \ref{fewfpowef23r2}. In addition, 
 we use that if $R$ is in $\cFin(Z)$ with the induced structures, then we have a coarse equivalence between $R$ and the subspace of $\pi_{0}(Z)_{min,max}$ of components intersecting $R$. This gives the marked equivalence in 
 \begin{eqnarray*}
E(Z\otimes X)&\simeq& F( Z\otimes X\otimes G_{can,min})\\&\stackrel{}{\simeq}&
  \colim_{R\in \cFin(Z)}F(R \otimes X\otimes G_{can,min})\\& 
  \stackrel{!}{\simeq} &\colim_{R\in \cFin(\pi_{0}(Z))}  F( R_{min,max}\otimes X\otimes  G_{can,min})\\&\stackrel{}{\simeq} &\colim_{R\in \cFin(\pi_{0}(Z))}R\otimes F( X\otimes G_{can,min})\\&\simeq&
\pi_{0}(Z)\otimes E(X)\ .   
\end{eqnarray*}

 \end{proof}

 Let  $\bC$ be a cocomplete stable $\infty$-category and $E$ be in $\Fun^{\colim}(G\Spc\cX,\bC)$.     Then $E^{\Bredon}$ is again in $\Fun^{\colim}(G\Spc\cX,\bC)$.    We  use  {the equivalence} \eqref{ewflkwenfjkewneiwjffefff21} in order to evaluate $E$ and $E^{\Bredon}$ on objects of $G\Sp\cX$.
Consider objects $T$  in $G\Orb$ and $X$  in $\Sp\cX$.
\begin{lem}\label{wefiweofewfewfwef} Assume 
 \textbf{one} of:
\begin{enumerate}
\item  $G$ is finite 
\item  $E$ is hyperexcisive   and $X\in \Sp\cX_{bd}$.
\end{enumerate} 
Then  the comparison map  $$\beta_{E,T_{min,max}\otimes X} \colon E^{\Bredon}(T_{min,max}\otimes X)\stackrel{\simeq}{\to} E(T_{min,max}\otimes X)\ $$
 {is an equivalence.}
\end{lem}
\begin{proof} 
If $G$ is finite, then
for   $S,R$ in $G\Orb$ by excision (used at $!$) we have  the equivalence
\begin{eqnarray*} E(S_{min,max}\otimes Y(T_{min,max}\otimes X)(R))&\simeq &E(S_{min,max}\otimes \Hom_{G\Orb}(R,T)_{min,max}\otimes X)\\&\stackrel{!}{\simeq}&
 \Hom_{G\Orb}(R,T)\otimes E(S_{min,max}\otimes X)\ .
\end{eqnarray*} 
 
Alternatively we  can obtain the same equivalence $!$ if we  assume that $E$ is hyperexcisive and $X$ is bounded since then also $S_{min,max}\otimes X$ belongs to $G\Sp\cX_{bd}$. 
%
We now use the {Yoneda} formula    in order to get 
 \begin{eqnarray*}
  E^{\Bredon}(T_{min,max}\otimes X)&\simeq&
  \int^{G\Orb } E((-)_{min,max}\otimes Y(T_{min,max}\otimes X)(-))
  \\&\simeq& \int^{G\Orb }   \Hom_{G\Orb}(-,T)\otimes E((-)_{min,max}\otimes X)\\&  \simeq & E( T_{min,max}\otimes X ) \ . \qedhere\end{eqnarray*}

\end{proof}

We consider a  cocomplete  {stable} $\infty$-category $\bC$ and a functor  $E\colon G\Orb\to \Fun^{\colim}(\Spc\cX,\bC)$.   Recall Definition \ref{etgioegggergergergeg}.

\begin{lem}\label{rgio34oergergergregeg} \mbox{}
{\begin{enumerate}
\item  If $G$ is finite, then  we have an equivalence \begin{equation}\label{rvoijoirff2efwef}
\underline{E^{G}\circ Y} \simeq E \ .
\end{equation}
\item \label{wuhfifefwfcqwefqwec}  {If}  $E$ takes values
in continuous  $\bC$-valued coarse homology theories, then  we have an equivalence \begin{equation}\label{rvoijoirff2efwef111}
\underline{E^{G}\circ Y}_{|\Spc\cX_{bd}} \simeq E_{|\Spc\cX_{bd}}\ .   
\end{equation}
\end{enumerate} }

\end{lem}
\begin{proof}
Note that by Lemma \ref{fewfpowef23r2}, applied to the case of a trivial group $G$, a continuous coarse homology theory is hyperexcisive.
 For $T$ in $G\Orb$ and $X $ in $\Spc\cX$ we calculate
\begin{eqnarray*}\underline{(E^{G}\circ Y)}(T)(X)&\simeq&
 E^{G}(Y(T_{min,max}\otimes X))\\&\simeq&
\int^{{S\in}G\Orb} E(S) ((T_{min,max}\otimes X)^{(S)})\\&\simeq& \int^{{S\in}G\Orb} E(S)(\Hom_{G\Orb}(S,T)_{min,max}\otimes X)\\&\stackrel{!}{\simeq}& \int^{{S\in}G\Orb} \Hom_{G\Orb}(S,T)\otimes   E(S)(  X)\\&\simeq&     E(T)( X)\ ,
\end{eqnarray*}
where at the marked equivalence we either used that $\Hom_{G\Orb}(S,T)$ is finite for finite $G$ or hyperexcisiveness of $E{(S)}$ and boundedness of $X$.
\end{proof}

 {Let $\bC$  a cocomplete stable $\infty$-category and  $E$ be in $\Fun^{\colim}(G\Spc\cX,\bC)$.}     The following is an immediate consequence of Lemma \ref{rgio34oergergergregeg} and  Definition \ref{rgfrihgioffwfwfwefewf}.
 Recall the comparison map from \eqref{f24oif}.
\begin{kor}\label{fiewofewefewfw}Assume  that $G$ is finite. 
Then the comparison map
$$ \beta_{E^{\Bredon}}:(E^{\Bredon})^{\Bredon} \simeq E^{\Bredon} $$
is an equivalence, i.e., $ E^{\Bredon}$ coincides with its associated
  Bredon-style equivariant coarse homology theory.
\end{kor}
\begin{proof} The comparison map is  equivalent to
$$(E^{\Bredon})^{\Bredon}\stackrel{Def. \ref{rgfrihgioffwfwfwefewf}}{\simeq} \underline{(\underline{E}^{G}\circ Y)}^{G}\circ Y\stackrel{\eqref{rvoijoirff2efwef}}{\simeq} 
\underline{E}^{G}\circ Y\stackrel{Def. \ref{rgfrihgioffwfwfwefewf}}{\simeq}E^{\Bredon}\ . $$
\end{proof}

\begin{ddd} \label{rgioregggegergergerg}We 
 let $$G\Sp\cX\langle G\Orb\otimes \Sp\cX\rangle$$ be the full stable subcategory of $G\Sp\cX$ generated under colimits by the motives   $S_{min,max}\otimes X$ for $S$ in $G\Orb$ and $X$ in $\Sp\cX$. 
 \end{ddd}

  {Let $\bC$   be  a cocomplete stable $\infty$-category,    let $E$ be in $\Fun^{\colim}(G\Spc\cX,\bC)$,   and consider~$X$ in $G\Sp\cX$.} The following is an immediate consequence of Lemma \ref{wefiweofewfewfwef}. 
\begin{kor}\label{erwfiowfwefewfwf}
Assume: 
\begin{enumerate}
\item  $G$ is finite{.} 
\item  $X$ belongs to $G\Sp\cX\langle G\Orb\otimes \Sp\cX\rangle$.
\end{enumerate} 
Then the comparison map
$$ \beta_{E,X}\colon E^{\Bredon}(X)\to E(X)$$ is an equivalence.
 \end{kor}
\begin{proof}
Here we use that the domain and the codomain of the comparison map preserve colimits and that the  {comparison} map is an equivalence  for the generators of $G\Sp\cX\langle G\Orb\otimes \Sp\cX\rangle$ by
Lemma  \ref{wefiweofewfewfwef}.  
\end{proof}

\begin{ex}\label{fioergwergewrgewrgregw}
Let $G$ be a group. Then for $S$ in $G\Orb$ we have   \begin{equation}\label{owjgopjkerpgferwgwegwergwergeg}
(G_{can,min})^{(S)}\simeq \left\{\begin{array}{cc}
  \emptyset&S\not\cong G\\
  G^{triv}_{can,max}&S\cong G
  \end{array}
\right.\ ,
\end{equation} 
where the superscript $triv$ indicates that $G_{can,max}$ is considered just as a bornological coarse {space} with the trivial $G$-action.
  Note that $G\cong \End_{G\Orb}(G)$ such that
 $g$ in $G$  acts on the orbit $G$ by  $h\mapsto hg^{-1}$.  By functoriality, the group $G^{op}$ acts on ${(}G_{can,min})^{(G)}$. Under the identification above the action of $g$ in $G^{op}$ on $G^{triv}_{can,max}$ is given by $h\mapsto g^{-1}h$.
 
  Note that the end  in Definition \ref{fuifehwieufhf23f2323rr233r2r3} can be expressed as a colimit over the twisted arrow category $ \Tw(G\Orb)^{op}$ of $G\Orb$. The full subcategory of $\Tw(G\Orb)^{op}$ of objects $T\to S$ with $S\cong G$
  is equivalent to $BG$ via the functor which sends the unique object $*$ of $BG$ to
  $G\stackrel{\id_{G}}{\to} G$, and $h$ in $G\cong \End_{BG}(*)$ to the square $$\xymatrix{G\ar[r]^{\id_{G}} \ar@{..>}[d]_{g\mapsto gh}&G \\G \ar[r]^{\id_{G}}&G\ar@{..>}[u]_{g\mapsto gh^{-1}}}$$
  (the dotted arrows are arrows in $G\Orb^{op}$, and the arrow in $\Tw(G\Orb)^{op}$ points downwards.)
 If $E$ is an equivariant coarse homology theory, we  then  get
 \begin{eqnarray*}E^{Bredon} (G_{can,min})&\simeq&  \int^{S\in G\Orb} E(S_{min,max}\otimes Y(G_{can,min})(S))\\&\simeq&
 \colim_{(T\to S)\in\Tw(G\Orb)^{op}} E(T_{min,max}\otimes Y(G_{can,min})(S))\\&\simeq&
 \colim_{BG} E(G_{min,max}\otimes G^{triv}_{can,max})\ ,
 \end{eqnarray*}
 where  the $G$-action (involved in the colimit) on $G_{min,max}\otimes G^{triv}_{can,max}$ is given by $(g,(h,h'))\mapsto (hg^{-1},gh)$.

The identity map of the underlying sets
 $$G_{min,max}\otimes G^{triv}_{can,max}\to G_{min,max}\otimes G^{triv}_{max,max}$$ is a continuous equivalence \cite[Def.~3.20]{equicoarse}, i.e.~it induces an equivalence in every continuous equivariant coarse homology theory.  
 Furthermore, the projection
 $$ G_{min,max}\otimes G^{triv}_{max,max}\to G_{min,max}$$ is a coarse equivalence compatible with the additional $G$-action given by $(g,h)\mapsto hg^{-1}$ on the target.  

  If $E$ is continuous, then  the composition of these two maps  induces the  equivalence     
  $$\colim_{BG} E(G_{min,max}\otimes G^{triv}_{can,max})\stackrel{\simeq}{\to}
  \colim_{BG} E(G_{min,max}) \ .$$
  Hence, assuming that $E$ is continuous, we have an equivalence 
  \begin{equation}\label{wergrwegrewffwerf}
E^{Bredon} (G_{can,min})\simeq \colim_{BG} E(G_{min,max})\ .
\end{equation}
{Under these assumptions, t}he comparison map has the form
  \begin{equation}\label{werjoiejvioervwervwervre}
\beta_{E,G_{can,min}}: \colim_{BG} E(G_{min,max})\to E(G_{can,min})\ .
\end{equation} \heB\end{ex}

  \begin{ex}\label{fkvsdfvfdvsdfvsdvs}
  
   We consider the  equivariant coarse algebraic $K$-theory  $K\C\cX^{G}$ 
 \cite[Def.~8.8]{equicoarse} associated to the  additive category  of finite-dimensional complex vector spaces. By   \cite[Prop.~8.17]{equicoarse}, it is continuous.

By \cite[Def.~8.25]{equicoarse}, we have an equivalence \begin{equation}\label{}
K\C\cX^{G}(G_{can,min})\simeq K(\C[G])\ .
\end{equation}

We now assume that $G$ is finite. Then, using  \cite[Lem.~8.20]{equicoarse}
and the fact that $G_{min,max}=G_{min,min}$    for the second equivalence, we have 
\begin{equation}\label{}
K\C\cX^{G,Bredon}(G_{can,min})\stackrel{\eqref{wergrwegrewffwerf}}{\simeq}
\colim_{BG}K\C\cX^{G}(G_{min,max}) \simeq \colim_{BG} K\C\simeq  K\C\otimes BG
\end{equation}
For the last equivalence we observe by inspection of the calculation  in \cite[Lem.~8.20]{equicoarse}  that the induced $G$-action on $K\C$ is trivial.

We now note that $K\C$ is connective, $\pi_{0}(K\C)\cong \Z$, and hence  $$\pi_{0}(K\C\cX^{G,Bredon}(G_{can,min}))\cong \pi_{0}(K\C\otimes BG)\cong \Z\ .$$
On the other hand,  for a finite group $G$ the group ring $\C[G]$ is Morita equivalent to the ring $\prod_{\hat G} \C$, where $\hat G$ is the set of isomorphism classes of  irreducible complex representations of~$G$. Therefore $ K(\C[G])\simeq \prod_{\hat G} K\C$, and hence $$\pi_{0}(K\C\cX^{G}(G_{can,min}))\cong \pi_{0}(K(\C[G]))\cong   \pi_{0}(\prod_{\hat G} K\C)\cong  \Z^{|\hat G|}\ .$$

By comparing the ranks of these homotopy groups we can conclude that the comparison map
$\beta_{K\C\cX^{G},G_{can,min}}$ is not  an isomorphism provided $G$ is non-trivial.   \heB
 \end{ex}

\begin{kor}\label{wtgijowgwrefwerfrfw} If $G$ is finite and non-trivial, then 
the inclusion $$G\Sp\cX\langle G\Orb\otimes \Sp\cX\rangle\to  G\Sp\cX$$ is proper. In particular,  $\Yo_{G}^{s}(G_{can,min})$ does not belong to the image.
\end{kor}
\begin{proof}
If, by contradiction, $\Yo_{G}^{s}(G_{can,min})$ belonged to $G\Sp\cX\langle G\Orb\otimes \Sp\cX\rangle$, then the comparison map $\beta_{E,G_{can,min}}$ would be an equivalence for every equivariant coarse homology theory $E$ by Corollary~\ref{erwfiowfwefewfwf} which contradicts the result of the calculation in Example \ref{fkvsdfvfdvsdfvsdvs} for $E=K\C \cX^{G}$.
\end{proof}

\color{black}

%

%
%

In order to formulate the analogue of Corollaries \ref{fiewofewefewfw} and \ref{erwfiowfwefewfwf} for infinite groups we introduce the following version  of Definition \ref{rgioregggegergergerg}.

  \begin{ddd}\label{zfikgloglugilgilk} We 
 let $$G\Sp\cX\langle G\Orb\otimes \Sp \cX_{bd}\rangle$$ be the full stable subcategory of $G\Sp\cX$ generated under colimits by the motives   $S_{min,max}\otimes X$ for $S$ in $G\Orb$ and $X$ in $\Sp\cX_{bd}$. 
 \end{ddd}
  {Let $\bC$ be a cocomplete stable $\infty$-category,   let  $E$ be in $\Fun^{\colim}(G\Spc\cX,\bC)$,   and consider~$X$ in $G\Sp\cX$.}
%
%
\begin{kor}\label{erwfiowfwefewfwf1}  Assume:
\begin{enumerate} \item   $E$  is hyperexcisive.
\item  $X$ belongs {to} $G\Sp\cX\langle G\Orb\otimes \Sp\cX_{bd}\rangle$.\end{enumerate} 
 Then the  comparison map
$$\beta_{E,X}\colon E^{\Bredon}(X)\to E(X)$$ is an equivalence.
 \end{kor}
\begin{proof}
The argument is the same as for Corollary \ref{erwfiowfwefewfwf}.
\end{proof}

\newcommand{\hlg}{\mathrm{hlg}}

\begin{rem}\label{fvoihweririjfeewfewfewfewfw}
The functor
$$G\Orb\to G\Sp\cX \ , \quad S\mapsto \Sigma \Yo^{{s}}_{G}(S_{min,max})$$
determines an equivariant homology theory 
$$\cO^{\infty}_{\hlg}\colon G\Top\to G\Sp\cX\ ,$$
see also 
\cite[Def.~10.10]{equicoarse}. 
 The following {lemma} shows that the subcategory $G\Sp\cX\langle G\Orb\otimes \Sp\cX_{bd}\rangle$
 is sufficiently rich.

\begin{lem}\label{ergbioeoreggrwrgwergwrgrg}
If $W$ is a $G$-CW complex, then $\cO^{\infty}_{\hlg}(W)$ belongs to $G\Sp\cX\langle G\Orb\otimes \Sp\cX_{bd}\rangle$.  
\end{lem} 
 \begin{proof}
We have $\Sigma \Yo^{{s}}_{G}(S_{min,max})\in G\Sp\cX\langle G\Orb\otimes \Sp\cX_{bd}\rangle$ for any $S$ in $G\Orb$. Then $\cO^{\infty}_{\hlg}(W)$ is a colimit of a diagram of objects in $ G\Sp\cX\langle G\Orb\otimes \Sp\cX_{bd}\rangle$ and therefore also contained in this subcategory.
 \end{proof}
 
 If $G$ is finite (or $E$ is hyperexcisive, respectively), then Corollary \ref{erwfiowfwefewfwf} (or Corollary  \ref{erwfiowfwefewfwf1}, respectively) applies to $\cO_{\hlg}^{\infty}(A)$ in place of $X$ for $A$ a $G$-CW complex. \hrB
 \end{rem}
 
%


We now consider some examples showing that the $E^{Bredon}\to E$ is far from being  an equivalence in general.
\begin{kor}\label{weijgowegrefweferf}
Assume \begin{enumerate}
\item $G$ is infinite.
\item $E$ is continuous.
\end{enumerate}
Then  $E^{Bredon}(G_{can,min})\simeq 0$. 
\end{kor}
\begin{proof}
If $G$ is infinite, then $G_{ min,max}$ does not admit non-empty $G$-invariant locally finite 
subsets. By continuity of $E$ we get
$E(G_{ min,max})\simeq 0$ and hence $E^{Bredon}(G_{can,min})\simeq 0$ by~\eqref{wergrwegrewffwerf}.
\end{proof}

 The comparison map $\beta_{E,G_{can,min}}$ is induced by the evaluation map \eqref{qeklnklnqklnwefqefqwefqefq} which under the identification  \eqref{owjgopjkerpgferwgwegwergwergeg} (applied to the second tensor factor of the domain)  is given by the map 
  $$G_{min,max}\otimes G^{triv}_{can,max}\to G_{can,min}\ , \quad  (h,h')\mapsto hh'\ .$$   Note that this map is compatible with the additional $G$-action on the domain by $(g,(h,h'))\to (hg^{-1},gh')$. It induces a map
$$E(G_{min,max}\otimes G^{triv}_{can,max})\to E(G_{can,min})$$
which descends to
$$\beta_{E,G_{can,min}}:\colim_{BG} E(G_{min,max}\otimes G^{triv}_{can,max})\to E(G_{can,min})\ .$$

\begin{rem}In the following we want to apply this formula to $E=F_{G_{can,min}}$ for a continuous $F$. Then $E$ is not continuous in general and we therefore can not simplify these formulas further like for the left-hand side  of \eqref{werjoiejvioervwervwervre}. \hrB
\end{rem}

In general we do not know how to calculate $E(G_{can,min})$, but for Abelian groups $G$ we can go further. The main point is the existence of the vertical maps in \eqref{vkjerhkjervervevw} below which for Abelian groups are morphisms of $G$-bornological coarse spaces.

For simplicity we specialize to the case $G=\Z$.  Recall that a morphism $f $ in a cocomplete stable $\infty$-category a is called a phantom morphism if  $f\circ \phi\simeq 0$ for all morphisms $\phi$ with compact domain.

\begin{lem}\label{tgegrwegergwerg}
Assume:
\begin{enumerate}
\item $G=\Z$.
\item $E=F_{G_{can,min}}$ for a continuous equivariant coarse homology theory $F$.
\end{enumerate}
Then
\begin{enumerate}
\item \label{wekogegwergwrgrvdfvs}
 $E(G_{can,min})\simeq \Sigma E(*) $.
 \item \label{ergoijoqwefqwefwqefqwefef}
 The comparison map
 $\beta_{E,G_{can,min}}\colon E^{Bredon}(G_{can,min})\to E(G_{can,min})$ is a phantom map. \end{enumerate}
\end{lem}
\begin{proof}
We have a commutative diagram in $G\BC$ \begin{equation}\label{vkjerhkjervervevw} 
 \xymatrix{G_{can,min}\otimes G_{min,max}\otimes G^{triv}_{can,max}\ar[rr]^-{ (k,h,h')\mapsto (k,hh')}\ar[d]_{\cong}^{(k,h,h')\mapsto (k,k^{-1}h,h')}&&G_{can,min}\otimes G_{can,min}\ar[d]_{\cong}^{(k,h)\mapsto (k,k^{-1}h)}\\G_{can,min}\otimes G^{triv}_{min,max}\otimes G^{triv}_{can,max}\ar[rr]^-{(k,h,h')\mapsto (k,hh')}&&G_{can,min}\otimes G^{triv}_{can,min}}\ .   \end{equation}
 
 For   every $G$-bornological coarse space $X$ and equivariant coarse homology theory $E$ we have an equivalence  $$E(X\otimes \Z_{can,min}^{triv})\simeq \Sigma E(X)$$ {obtained by excision, see} \cite[Ex.~4.9]{buen}. 
 The right vertical isomorphism in \eqref{vkjerhkjervervevw} now provides the first   equivalence in
 $$
E(G_{can,min}) \simeq  E(G^{triv}_{can,min}) \simeq  \Sigma E(*)$$
proving \ref{wekogegwergwrgrvdfvs}.
 
 Since $F$ is continuous and in view of \eqref{vkjerhkjervervevw} the comparison map $\beta_{E,G_{can,min}}$ is equivalent to the map
 $$\colim_{BG}\colim_{L} F(L)\to E(  G_{can,min}^{triv})\ ,$$
 where $L$ runs over the poset of invariant locally finite subsets of $G_{can,min}\otimes G^{triv}_{min,max}\otimes G^{triv}_{can,max}$.
 
 Let $L$ be a locally finite invariant subset of $G_{can,min}\otimes G^{triv}_{min,max}\otimes G^{triv}_{can,max}$. Then $L= G\times L_{0}$ for some finite subset $L_{0}$ of $G^{triv}_{min,max}\otimes G^{triv}_{can,max}$.
 Hence the comparison map is equivalent to $$\colim_{BG}\colim_{L_{0}} E(L_{0})\to E(  G_{can,min}^{triv})\ ,$$
 where now $L_{0}$ runs over the poset of finite subsets of $G^{triv}_{min,max}\otimes G^{triv}_{can,max}$.

Recall that $G^{triv}_{min,max}\otimes G^{triv}_{can,max}$ has an additional $G$-action by $g(h,h'):=(hg^{-1},gh')$.
We let $\cF$ denote the poset of $G$-invariant subsets of the form $G L_{0}$ for finite subsets $L_{0}$ as above.

We can now rewrite by a cofinality 
{argument} the domain of the comparison map as
\begin{eqnarray*}\colim_{BG}\colim_{L_{0}}   E(L_{0})&\simeq& \colim_{BG} \colim_{Q\in \cF} \colim_{L_{0}\subseteq Q} E(L_{0})\\&\simeq&   \colim_{Q\in \cF}  \colim_{BG} \colim_{L_{0}\subseteq Q} E(L_{0})\ .
\end{eqnarray*}
The multiplication map $G^{triv}_{min,max}\otimes G^{triv}_{can,max}\to G_{can,min}^{triv}$ sends $Q$ in $\cF$ to a finite subset of~$G_{can,min}^{triv}$. This finite subset is  contained {in} some flasque subspace $R:=[n,\infty)$ of $G^{triv}_{can,min}$ for some sufficiently small $n$ in $\Z$. For fixed $Q$ the canonical morphism therefore has a factorization 
$$  \colim_{BG} \colim_{L_{0}\subseteq Q} E(L_{0})\to \colim_{BG}E(R)\to E(G_{can,min}^{triv})$$ and hence vanishes. The comparison map itself is then a filtered colimit  over $Q$ in $\cF$ of zero maps and hence a phantom map. This finishes the proof of \ref{ergoijoqwefqwefwqefqwefef}. 
\end{proof}

\begin{kor}\label{weruhigwergrefwfref}  If $G=\Z$, then  
$\Yo_{G}^{s}(G_{can,min})$ does not belong to $G\Sp\cX\langle G\Orb\otimes \Sp\cX_{bd}\rangle$.
\end{kor}
\begin{proof} Let $G:=\Z$.
 There exists a continuous  equivariant coarse homology theory $F$ such that $F(G_{can,min})$ is not
a phantom object, i.e., $\id_{F(G_{can,min})}$ is not a phantom morphism. We can e.g.~choose $F=K\C\cX^{G}$. We then set $E:=F_{G_{can,min}}$. Then being a phantom morphism by Lemma \ref{tgegrwegergwerg}.\ref{ergoijoqwefqwefwqefqwefef} the  comparison morphism $E^{Bredon}(G_{can,min})\to E(G_{can,min})\simeq \Sigma E(*)$ is not an equivalence.  Since $E$ is hyperexcisive  by Lemma   \ref{fewfpowef23r2},  in view of Lemma \ref{erwfiowfwefewfwf1}  $\Yo_{G}^{s}(G_{can,min})$ does not belong to $G\Sp\cX\langle G\Orb\otimes \Sp\cX_{bd}\rangle$.
  \end{proof}

 \subsection{Coarse Assembly Maps}\label{greoijio3gergrgregergerg}
\newcommand{\Ass}{\mathrm{Asbl}}

  Let $\cF$ be a {conjugation invariant} set  of   subgroups of $G$,  $\bC$ be a cocomplete $\infty$-category, and consider 
  a functor 
 $E\colon G\Orb\to \Fun^{\colim} (\Spc\cX,\bC)$. Using the functors from the adjunction \eqref{fwrelkoelfwefwef}  we define
$$E_{\cF}:=\Ind_{\cF}\Res_{\cF}E:G\Orb\to \Fun^{\colim} (\Spc\cX,\bC)\ .$$
The counit of the adjunction \eqref{fwrelkoelfwefwef}  provides a natural transformation
  \begin{equation}\label{vrevih3oif3rf3efrc}
E_{\cF}\to E\ .
\end{equation}
 \begin{ddd}\label{rgkierog43grgreg}For $X$ in $\Fun(G\Orb^{op},\Spc\cX)$ the morphism
$$\alpha_{E,\cF,X}\colon E^{G}_{\cF}(X)\to E^{G}(X)$$ induced by \eqref{vrevih3oif3rf3efrc}
is called the assembly map.\end{ddd}
\begin{rem}
The assembly map defined in \ref{rgkierog43grgreg} is the bornological-coarse analogue of the {Davis-L\"uck} assembly map  considered
in the study of isomorphism conjectures like the Baum-Connes or {Farrell}-Jones conjecture. \hrB
\end{rem}

One can ask under which conditions on $\cF$, $X$, and $E$ the assembly map is an equivalence.

\begin{ex}
If   $X\simeq \Ind_{\cF}\Res_{\cF}X$, then if follows from an analogue of \eqref{ewrgergrefwf} with $\cF^{\perp}$ replaced by $\cF$  that $\alpha_{E,\cF,X}$ is an equivalence. \heB
\end{ex}
 
 Let $\bC$ be  in addition stable  stable,    $E$ be in $\Fun^{\colim}(G\Sp\cX,\bC)$, and  $\cF$ be  a 
family   of subgroups  {$\cF$} of $G$.  Then we have {the following}  two stages of approximations of $E$
$$ \sigma_{E,\cF}:E^{\Bredon}_{\cF}\stackrel{\alpha_{\underline{E},\cF}\circ Y}{\to} E^{\Bredon}\stackrel{\beta_{E}}{\to} E\ .$$
We first approximate $E$ by its associated Bredon-style equivariant coarse homology theory $E^{\Bredon}$, and then we approximate the latter by its restriction $E^{\Bredon}_{\cF}$ to the family $\cF$.
The quality of the approximation $\beta_{E}$ of $E$ by its associated Bredon-style homology theory has been discussed in Section \ref{f42oifj2o3f32f32e}. In the present section we consider the  approximation
$\alpha_{E,\cF}:=\alpha_{\underline{E},\cF}$. 

If $H$ is a subgroup of $G$, then
 we have  the induction  functor
$$\Ind_{H}^{G}\colon H_{\cF\cap H}\Orb\to G_{\cF}\Orb \ , \quad U\mapsto G\times_{H}U\ .$$     
For $E$ as above and for every subgroup $H$ of $G$ we then define the functor 
\begin{equation}\label{veroi3joi4ferfreve}
E_{H}\colon H\Orb \to \bC, \quad E_{H}(S):=E( \Ind_{H}^{G}(S)_{min,max})\ .
\end{equation}   For any functor $F$ defined on $H\Orb$ with values in some cocomplete target, and  for any family of subgroups $\cF^{\prime}$ of $H$, we 
 have  the Davis-L\"uck type assembly map
\begin{equation}\label{dsfvfdsvfdverwvvsvdsfvsfdvsfv}  \Ass_{F,\cF^{\prime} }\colon \colim_{U\in H_{\cF^{\prime}  }\Orb} F(U)\to  F(*)  \end{equation} 
induced by the collection of morphisms $U\to *$ for all $U$ in $ H_{\cF'}\Orb $. 


{Let $\bC$ be a cocomplete stable $\infty$-category,   let
  $E$ be in $\Fun^{\colim}(G\Spc\cX,\bC)$, and let  $X$ be in $G\Sp\cX$.} 
  By $G\Sp\cX_{bd}$ we denote the full stable subcategory of $G\Sp\cX$ generated under colimits by motives bounded $G$-bornological coarse spaces.
\begin{prop}\label{efgijofdewfwefewfwf1}Assume:\begin{enumerate}
\item $E$ is hyperexcisive.
\item $X$ belongs to $G\Sp \cX_{bd}${.} \item \label{frelkn4lk3refrefeer1}
The assembly map $\Ass_{E_{H},\cF\cap H}$ is an equivalence for every subgroup $H$ of $G$.
   \end{enumerate} 
   Then 
$$\alpha_{E,\cF,X}  \colon E^{\Bredon}_{\cF}(X)\to  E^{\Bredon} (X)$$
is an equivalence.
 \end{prop}
\begin{proof}
 

For a subgroup $H$ of $G$ the 
 induction   functor induces an equivalence of categories
$$\Ind_{H}^{G}\colon H_{\cF\cap H}\Orb\to G_{\cF}\Orb_{/G/H}\ , \quad U\mapsto(G\times_{H}U\to G/H)\ .$$ 
 Indeed, an inverse functor (with obvious transformations witnessing this fact)  is given by
$$ G_{\cF}\Orb_{/G/H} \to H_{\cF\cap H}\Orb \ , \quad (p:T\to G/H)\mapsto  p^{-1}(eH)\ .$$
Therefore the Assumption \ref{frelkn4lk3refrefeer}. can be rewritten  as asserting the equivalence \begin{equation}\label{rgkjwnwjkefnwefwefew}
\colim_{(R\to S)\in G_{\cF}\Orb _{/S}} E(R_{min,max})\stackrel{\simeq}{\to} E(S_{min,max})
\end{equation}
 for every $S$ in $G\Orb$.

Since the domain and the codomain of $\alpha_{E,\cF}$ preserve colimits it suffices to show the equivalence for bounded $G$-bornological coarse spaces $X$.  If $X$ is bounded, then $X^{(S)}$  is bounded  for every 
$S$ in $G\Orb$.  In view of Definition \ref{rgfrihgioffwfwfwefewf} it therefore suffices to show that
$\underline{E}_{\cF}(S)(W)\to \underline{E}(S)(W)$ is an equivalence for every
$S$ in $G\Orb$ and bounded bornological coarse space $W$. 
This follows from the chain of equivalences    
     \begin{eqnarray*}
     \underline{E}_{\cF}(S)(W)&\simeq&(\Ind_{\cF}\Res_{\cF}\underline{E})(S)(W)\\&\simeq&
     \colim_{(R\to S)\in G_{\cF}\Orb}\underline{ E}(R)(W)\\&\simeq&
     \colim_{(R\to S)\in G_{\cF}\Orb}E(R_{min,max}\otimes W)\\&\stackrel{ \eqref{fv4oijioejioccowicjwecwcweccwecwecwec}
}{\simeq}&
 \colim_{(R\to S)\in G_{\cF}\Orb}  \pi_{0}(W)\otimes E(R_{min,max})\\&\stackrel{\eqref{rgkjwnwjkefnwefwefew}}{\simeq}&
  \pi_{0}(W)\otimes E(S_{min,max})\\
  &\stackrel{ \eqref{fv4oijioejioccowicjwecwcweccwecwecwec}
}{\simeq}&
   E(S_{min,min}\otimes W)\\
   &\simeq&
   \underline{E}(S)(W)\ ,
          \end{eqnarray*}
          where we use the assumption that $E$ is hyperexcisive in order to be able  apply  \eqref{fv4oijioejioccowicjwecwcweccwecwecwec}.
\end{proof}
%
%
%
%
%
%

{Let $\bC$ be a cocomplete stable $\infty$-category,   let
  $E$ be in $\Fun^{\colim}(G\Spc\cX,\bC)$, and let  $X$ be in $G\Sp\cX$.} 
Combining Proposition \ref{efgijofdewfwefewfwf1} and Corollary \ref{erwfiowfwefewfwf1}, and using that
$G\Sp\cX\langle G\Orb\otimes \Sp\cX_{bd }\rangle$ is contained in $G\Sp\cX_{bd}$, we get:
\begin{kor}\label{efgijofdewfwefewfwf}Assume:\begin{enumerate}
\item $E$ is hyperexcisive.
\item $X$ belongs to $G\Sp\cX\langle G\Orb\otimes \Sp\cX_{bd}\rangle$.
\item \label{frelkn4lk3refrefeer}

The assembly map $\Ass_{E_{H},\cF\cap H}$ is an equivalence for every subgroup $H$ of $G$.
   \end{enumerate} 
   Then 
$$ \sigma_{E,\cF,X}  : E^{\Bredon}_{\cF}(X)\to  E (X)$$
is an equivalence.
 \end{kor}

 \begin{ex}\label{beoijewoigfergreree}We consider the continuous equivariant coarse topological $K$-homology functor $K\cX^{G}$ in $\Fun^{\colim}(G\Spc\cX,\Sp)$ constructed in \cite[Thm.~6.3]{bu}. We thereby  use    the  topological $K$-theory of $C^{*}$-categories from \cite[Def.~14.3]{cank} in place of the homological functor and the   $C^{*}$-category $\Hilb_{c}(\C)$  with trivial $G$-action of Hilbert spaces and compact operators   as coefficient $G$-$C^{*}$-category.
We define \begin{equation}\label{vsdfvsdfvfsvwref} KU_{r,G}\colon G\Orb\to \Sp\ , \quad S\mapsto K\cX_{G_{can,min}}^{G}(S_{min,max})\ .\end{equation}

Then we have equivalences 
$$KU_{r,G}(G/H)\simeq 
  K_{\Calg}(C_{r}^{*}(H))$$
 for all subgroups $H$ of  $G$, where   $K_{\Calg}$ is the spectrum-valued $K$-theory functor for $C^{*}$-algebras e.g.~defined as $K_{\Calg}(-):=\KK(\C,-)\colon\nCalg\to \Sp$, see Remark \ref{erthopetrghertge}.
 
 By \cite[Cor.~10.5]{bu}, we have an equivalence \begin{equation}\label{wfefwefweccsc} K\cX^{H}\simeq K\cX^{G}\circ \Ind_{H}^{G}\end{equation} of functors from $H\BC$ to $\Sp$,
 where $\Ind_{H}^{G}\colon H\BC\to G\BC$ is the induction functor for bornological coarse spaces.
 For every $S$ in $H\Orb$  we have a canonical continuous equivalence
 $$ \Ind_{H}^{G}(H_{can,min}\otimes S_{min,max})\to G_{can,min}\otimes \Ind_{H}^{G}(S)_{min,max}$$
 which on the underlying $G$-sets is given by
 $[g,(h,s)]\mapsto (gh,[g,s])$. It induces the second   in  the chain of equivalences
 \begin{equation}\label{vfvdfvsvvs}  K\cX^{H}_{H_{can,min}}\circ (-)_{min,max}\stackrel{\eqref{wfefwefweccsc}}{\simeq} ( K\cX^{G}\circ  \Ind_{H}^{G})_{H_{can,min}}(-)_{min,max}\stackrel{\simeq}{\to}  K\cX^{G}_{G_{can,min}} (\Ind_{H}^{G}(-)_{min,max})  \end{equation}of functors from $H\Orb$ to $\Sp$.
 Specializing  
  \eqref{veroi3joi4ferfreve}      we get the equivalence
 $$(K\cX_{G_{can,min}}^{G})_{H}(-)\stackrel{\eqref{vfvdfvsvvs}}{\simeq} K\cX^{H}_{H_{min,max}}\circ (-)_{min,max}\stackrel{\eqref{vsdfvsdfvfsvwref}}{\simeq} KU_{r,H}\ .$$
 The Baum-Connes conjecture for the group $H$ asserts, that the assembly map
 $\Ass_{KU_{r,H},\Fin}$   is an equivalence, where $\Fin$ denotes the family of finite subgroups.
 
 \begin{rem}
 Note that the Baum-Connes conjecture is classically formulated for the Kasparov assembly map.
 Since we now know  by \cite{kranz} that the  Kasparov assembly for the group $H$ (with trivial coefficients) is equivalent to the Davis-L\"uck assembly map $\Ass_{KU_{r,H},\Fin}$ in \eqref{dsfvfdsvfdverwvvsvdsfvsfdvsfv} we can safely consider
 the latter as the assembly map featuring the Baum-Connes conjecture. The identification of the functor 
$KU_{r,H}$ with the equivariant $K$-homology functor considered in  \cite{kranz} is discussed in \cite[Sec.~16]{bel-paschke}.
 \hrB
\end{rem}

The equivariant coarse homology theory $K\cX^{G}$ is continuous, and   consequently $K\cX_{G_{can,min}}^{G}$ is {hyperexcisive} by Lemma \ref{fewfpowef23r2}.
 The Corollary \ref{efgijofdewfwefewfwf} now has the following immediate consequence.
  \begin{kor} \label{wrtijgowrtfewrfwerfw}Assume:
  \begin{enumerate}
 \item\label{weogkpwtgewrgerfw}  All subgroups of $G$
satisfy the   Baum-Connes conjecture for trivial coefficients.
 \item $X$  belongs to $G\Sp\cX\langle G\Orb\otimes G\Sp\cX_{bd}\rangle$. \end{enumerate}
  Then
 $$\sigma_{K\cX_{G_{can,min}}^{G},\Fin,X}\colon  (K\cX^{G}_{G_{can,min}})^{\Bredon}_{\Fin}(X)\to K\cX_{G_{can,min}}^{G}(X)$$
 is an equivalence.  
 \end{kor}
 \begin{rem}
 If the Baum-Connes conjecture with coefficients holds for $G$, then all subgroups of $G$ also 
 satisfy the Baum-Connes conjecture with coefficients   \cite[Thm.~2.5]{MR1836047}. Therefore, we could get a simpler-looking statement by  replacing  Assumption \ref{wrtijgowrtfewrfwerfw}.\ref{weogkpwtgewrgerfw} by the  stronger condition
 \item[\mbox{}\hspace{0.3cm} \textit{1'}.] {\em $G$ satisfies the Baum-Connes conjecture  with coefficients.}
\hrB \end{rem}
 \heB
\end{ex}

\begin{ex}\label{rgeroij4oi3ergergregergreg} This example is completely parallel to Example \ref{beoijewoigfergreree}. Let $\bA$ be an additive  category   and $K\bA\cX^{G}$ in $\Fun^{\colim}(G\Spc\cX,\Sp)$ be the continuous equivariant coarse algebraic $K$-theory functor  with coefficients in  $\bA$ (see \cite[Def.~8.8]{equicoarse}).
The values of its twisted version are given by
$$K\bA\cX_{G_{can,min}}^{G}(G/H)\simeq K \bA^{G}(G/H)\simeq  K^{alg}(\bA[H]) \ ,$$
where $K^{alg}$ is the non-connective $K$-theory functor for additive categories, and $K \bA^{G}:G\Orb\to \Sp$ is the functor  first defined by Davis-L\"uck \cite{MR1659969}.
We again have the equivalences
$$K\bA\cX_{G_{can,min}}^{G} ( \Ind_{H}^{G}(-)_{min,max})\simeq K\bA\cX_{H_{can,min}}^{H}((-)_{min,max})$$
of functors from $H\Orb$ to $\Sp$ for all subgroups $H$ of $G$. Furthermore we can identify the functor  \eqref{veroi3joi4ferfreve} 
by
$$(K\bA\cX_{G_{can,min}}^{G})_{H}\simeq K\bA^{H}\ .$$
The {Farrell}-Jones conjecture for $H$ with {coefficients} in $\bA$ asserts that the assembly map $\Ass_{K \bA^{H},\mathcal{VC}yc}$ is an equivalence, where $ \mathcal{VC}yc$ denotes the family of virtually cyclic subgroups of $H$.

The equivariant coarse homology theory $K\bA\cX^{G}$ is continuous, and  consequently $K\bA\cX_{G_{can,min}}^{G}$ is hyperexcisive by Lemma \ref{fewfpowef23r2}. The Proposition \ref{efgijofdewfwefewfwf}  has the following immediate consequence.
\begin{kor} Assume:
  \begin{enumerate} \item
The  {Farrell}-Jones   conjecture with coefficients in $\bA$ holds  true for all  subgroups of $G$. 
\item  $X$  belongs to $G\Sp\cX\langle G\Orb\otimes G\Sp\cX_{bd}\rangle$. \end{enumerate}
Then 
 $$\sigma_{ K\bA\cX_{G_{can,min}}^{G},\mathcal{VC}yc,X}\colon  (K\bA\cX^{G}_{G_{can,min}})^{\Bredon}_{\mathcal{VC}yc}(X)\to K\bA\cX_{G_{can,min}}^{G}(X)$$
 is an equivalence. 
 \end{kor}
   \heB
 \end{ex}

 \section{Localization for coarse homology theories}

 \subsection{The Coarse Abstract Localization Theorems}
 
 In this section we state and prove the  
  abstract localization theorems  for  Bredon-style equivariant coarse homology theories.

   Let $X$ be an object of $\Fun(G\Orb^{op},\Spc\cX)$, and let $\cF$  be a conjugation invariant set of {subgroups} of $G$.  We define  the morphism $X^{\cF}\to X$ as in Definition \ref{rgijofewrgergergerg}.

  We {let $\bC$ be a   cocomplete stable $\infty$-category and} consider a functor  $E\colon G\Orb\to \Fun^{\colim}(\Spc\cX,\bC)$. {We furthermore}  recall the Definition \ref{fuifehwieufhf23f2323rr233r2r3} of the functor $E^{G}$.

 \begin{theorem}[Coarse  Abstract Localization Theorem I]\label{fueiwfhuiwefewf23r23f}
Assume:
\begin{enumerate}
\item $\cF$ is a family of subgroups of $G$.
\item  \label{riothrhegtrgegeg}$E$ vanishes on $\cF$.
\end{enumerate}
Then the induced morphism $$E^{G}( X^{\cF})\to E^{G}(X)$$
  is an equivalence.
\end{theorem}

 \begin{proof}
 The proof is analogous to the one of Theorem \ref{wregfio324r3t3t4t}.
 The Assumption \ref{fueiwfhuiwefewf23r23f}.\ref{riothrhegtrgegeg} implies that $\Ind_{\cF^{\perp}}\Res_{\cF^{\perp}}E\to E$ is an equivalence.
 We then use the equivalence  \eqref{ewrgergrefwf} and Definition~\ref{rgijofewrgergergerg}  in order to conclude the assertion. 
 %
%
%
\end{proof}
   

 The second version of the abstract localization theorem starts with the choice of a 
 conjugacy class $\gamma$ of $G$.
Let $X$ be a $G$-bornological coarse space.
\begin{ddd} We define  the $G$-bornological coarse space $X^{\gamma}$  to be the  $G$-invariant subset
$$X^{\gamma}:=\bigcup_{g\in \gamma} X^{g}$$ of $X$ with the induced bornological coarse structures.
\end{ddd}

If $X$  is a $G$-bornological coarse space and $A$ is a $G$-invariant subset of $X$, then we {denote}
by $\{A\}:=  (U[X])_{U\in \cC^{G}}$ the big family generated by $A$, i.e.,   the family $(U[X])_{U\in \cC^{G}}$ of $U$-thickenings of $A$ for all $G$-invariant coarse entourages $U$ of $X$.  If $F$ is a functor defined on  $G\BC$ with a cocomplete target, then we define $F(\{A\}):=\colim_{A^{\prime}\in \{A\}} F(A^{\prime})$.  We have   natural morphisms \begin{equation}\label{rrgkegoergerggerwtrger} F(A)\to
F(\{A\})\to F(X) \ .
\end{equation}

%

We consider a cocomplete stable $\infty$-category $\bC$, a functor  $E\colon G\Orb\to \Fun^{\colim}(\Spc\cX,\bC)$, and we let $X$ be a  $G$-bornological coarse space.
We furthermore let  $\gamma$ be a conjugacy class of $G$. Recall   the functor
$\tilde Y\colon G\BC\to \Fun(G\Orb^{op},\Spc\cX)$ {defined as the composition  \eqref{f34fpo4fr3fg3f}.}
\begin{theorem}[Coarse Abstract Localization Theorem II]\label{rgioewfwefewfewfewf} Assume: 
\begin{enumerate}
\item\label{qkrgkoewgjwoirgrewfwrfw} $\Yo^{s}_{G}(X^{\gamma})\to \Yo^{s}_{G}(\{X^{\gamma}\}) $ is an equivalence.
\item $E$ vanishes on  $\cF(\gamma)$.
\end{enumerate}
Then the  map
$$E^G( X^{\gamma} )\to E^{G}( X )$$  induced by the inclusion $X^{\gamma}\to X$ is an equivalence.
\end{theorem}
\begin{proof}
We prepare  the proof of  the theorem 
{with the following  two} 
 lemmas.   
\begin{lem}\label{3rgorg34g34g111} {If  $X^{\gamma}=\emptyset$, then 
we have
$\tilde Y(X)^{\cF(\gamma)}\simeq \emptyset$.} 
\end{lem}
\begin{proof}
Assume that $X^{\gamma}=\emptyset$. If $H$ is in $\cF(\gamma)^{\perp}$, then there exists $h$ in $H\cap \gamma$. But then
$X^{H}\subseteq  X^{h}\subseteq X^{\gamma}$, i.e., $ X^{H}=\emptyset$. This implies that $\tilde Y(X)(G/H)\simeq \emptyset$. 
We conclude that
$\Res_{\cF(\gamma)^{\perp}}\tilde Y(X)\simeq \emptyset$ and hence $\tilde Y(X)^{\cF(\gamma)}\simeq \emptyset$ by  Definition \ref{rgijofewrgergergerg}.

%
\end{proof}
%
%
\begin{lem}\label{fuhi23f23f23f32f2f}
If $X^{\gamma}=\emptyset$ and $E$ vanishes on
$\cF(\gamma)$, then
$E^{G}(X)\simeq 0$.
\end{lem}
\begin{proof}
This is an immediate consequence of Theorem \ref{fueiwfhuiwefewf23r23f} and Lemma  \ref{3rgorg34g34g111}.
\end{proof}

%
%
We consider the complementary pair
$(X\setminus X^{\gamma},\{X^{\gamma}\})$.
By excision for the equivariant coarse homology theory $E^{G}\circ \tilde Y$ (see Corollary \ref{gih42iuhjfierfwfef} and \cite[Def.~3.10]{equicoarse}) we have the {push-out square \begin{equation}\label{weedewfdwesdcsdcsewdwedd}
\xymatrix{E^{G}( \{X^{\gamma}\}\cap (X\setminus X^{\gamma}) )\ar[r]\ar[d]&E^{G}( \{X^{\gamma}\} )\ar[d]\\E^{G}( X\setminus X^{\gamma})\ar[r]&E^{G}( X)}\ .
\end{equation}}
By {Lemma} \ref{fuhi23f23f23f32f2f} we have
$E^{G}( X\setminus X^{\gamma})\simeq 0$ {and $E^{G}( \{X^{\gamma}\}\cap (X\setminus X^{\gamma}))\simeq 0$}. 
{Hence the right vertical map in \eqref{weedewfdwesdcsdcsewdwedd} is an equivalence.} By Assumption  \ref{rgioewfwefewfewfewf}.\ref{qkrgkoewgjwoirgrewfwrfw}   we have an equivalence
$E^{G}(  X^{\gamma}) \simeq E^{G}( \{X^{\gamma}\})$ induced by the inclusions of $X^{\gamma}$ into   the members of $\{Y^{\gamma}\}$.
Both facts together imply the assertion of the theorem.
\end{proof}

If in Theorem \ref{rgioewfwefewfewfewf} we drop  Assumption  \ref{rgioewfwefewfewfewf}.\ref{qkrgkoewgjwoirgrewfwrfw}, then the same argument still proves the following version:
\begin{theorem}[Coarse Abstract Localization Theorem II${}^{\prime}$]\label{rgioewfwefewfewfewfppp} If  
 $E$ vanishes on $\cF(\gamma)$, then
  the second   natural morphism  in \eqref{rrgkegoergerggerwtrger}
$$E^G( \{X^{\gamma}\} )\to E^{G}( X)$$ is an equivalence.
\end{theorem}

We now consider a $\bC$-valued equivariant coarse homology theory, i.e., a  functor  $E $ in $ \Fun^{\colim}(G\Spc\cX,\bC)$.
As a corollary  of Theorem \ref{rgioewfwefewfewfewf} we obtain its  version for the associated Bredon-style
equivariant coarse homology of $E$. Recall the Definition {\ref{etgioegggergergergeg}} of the functor~$\underline{E}$.
Let  $X$ be a  $G$-bornological coarse space and 
   $\gamma$ be a conjugacy class of $G$. 
   
   \begin{kor}\label{gerhgiu34hiu34huferererg}
   Assume:
   \begin{enumerate}
\item $\Yo^{s}_{G}(X^{\gamma})\to \Yo^{s}_{G}(\{X^{\gamma}\}) $ is an equivalence.
\item $\underline{E}$ vanishes on $\cF(\gamma)$.
\end{enumerate}
Then the  morphism
$$ E^{\Bredon}(   X^{\gamma}  )\to E^{\Bredon}( X )$$  induced by the inclusion $X^{\gamma}\to X$ is an equivalence.
 \end{kor}

Similarly, from Theorem \ref{rgioewfwefewfewfewfppp} we get:
 \begin{kor}  
  If $\underline{E}$ vanishes on $\cF(\gamma)$, then
  the  morphism
$$ E^{\Bredon}(   \{X^{\gamma}\}  )\to E^{\Bredon}( X )$$  is an equivalence.
 \end{kor}

\subsection{Localization for equivariant coarse topological $K$-homology}\label{greiuhiufhiwefwefewfewf2r323r}
 
In this section we show how the Abstract Coarse  Localization Theorem II \ref{rgioewfwefewfewfewf} and its Corollary~\ref{gerhgiu34hiu34huferererg}, in particular, can be applied in the case of the equivariant coarse  topological $K$-homology $K\cX^{G}$ from  \cite[Thm.~6.3]{bu} discussed in Example~\ref{beoijewoigfergreree}. 
Note that the category $G\BC$
has a symmetric monoidal structure denoted by $\otimes$.     The category $G\Spc\cX$ {has} a symmetric monoidal structure essentially uniquely determined by the property that
the functor $\Yo\colon G\BC\to G\Spc\cX$ has a  symmetric monoidal refinement \cite[Lem.~4.16]{equicoarse}. 
On $\Sp$ we consider the symmetric monoidal structure given by the wedge product.

\begin{prop}The functor $K\cX^{G}\colon G\BC\to \Sp$ has a lax symmetric monoidal refinement.\end{prop}
\begin{proof}

Recall that this functor is defined as the composition 
$$K\cX^{G}:G\BC\xrightarrow{\bV^{G}}  C^{*}\mathbf{Cat}\xrightarrow{\mathrm{kk}_{\Ccat}} \KK\xrightarrow{{\KK}(\C,-)}\Sp\ .$$ 
 The functor $\mathrm{kk}_{\Ccat}$ defined in \cite[Def. 1.29]{KKG}
has a symmetric monoidal refinement by \cite[Thm.~1.35]{KKG}.
The functor $\KK(\C,-)$ has a lax symmetric monoidal refinement  since $\C$, being the tensor unit of $\KK$, has the structure of a cocommutative coalgebra object. It thus suffices to construct a lax symmetric monoidal refinement of the functor $\bV^{G}$\footnote{In the notation of \cite{bu} this is the functor $\bar \bC^{G,\mathrm{crt}}_{\mathrm{lf}}$ from \cite[Def.~4.19.2]{bu} for $\bC=\Hilb_{c}(\C)$ with the trivial $G$-action. } which sends~$X$ in $G\BC$ to the $C^{*}$-category of $X$-controlled $G$-objects in $\Hilb_{c}(\C)$.

Since tensor products of Hilbert spaces are only defined up to canonical isomorphisms we
actually only get a symmetric monoidal refinement of the composition
$$\bV^{G}_{\infty}:G\BC\xrightarrow{\bV^{G}}  C^{*}\mathbf{Cat}  \xrightarrow{\ell}  C^{*}\mathbf{Cat}_{\infty}\ ,$$ 
where $\ell$ presents  the localization of $C^{*}\mathbf{Cat}$  at unitary the equivalences. But this suffices in view of the commutative  diagram
$$\xymatrix{& C^{*}\mathbf{Cat}\ar[dd]^{\ell}\ar[dr]^{\mathrm{kk}_{\Ccat}}&\\G\BC \ar[ur]^{\bV^{G}}  \ar@{-->}[dr]_{\bV^{G}_{\infty}}&&\KK\\ & C^{*}\mathbf{Cat}_{\infty}\ar@{..>}[ur]_{\mathrm{kk}_{\Ccat,\infty}}&}\ .$$
The localization $\ell$ admits a symmetric monoidal refinement since the tensor product  with a fixed $C^{*}$-category preserves unitary equivalences. 
The  factorization  $ \mathrm{kk}_{\Ccat,\infty}$ exists and admits a symmetric monoidal refinement since 
$\mathrm{kk}_{\Ccat}$ sends unitary equivalences to equivalences by \cite[Thm.~1.32.2]{KKG}.

In order to construct the symmetric monoidal refinement of $\bV_{\infty}^{G}$ in a coherent way we proceed as in the algebraic situation considered in  \cite[Sec.~3.4]{symm}. 
At this point we only discuss some analytic details of the tensor product of objects and morphisms which are not present in the algebraic case.

Let $X,X'$ be in $G\BC$ and consider objects 
 $(H,\rho,\mu)$  in $\bV^{G}(X)$ and $(H',\rho',\mu')$ in $\bV^{G}(X')$, see \cite[Sec.~4]{bu} for notation.
 Here $H,H'$ are Hilbert spaces, $\rho,\rho'$ are unitary representations of $G$ on the respective Hilbert spaces, and $\mu,\mu'$ are projection-valued measures.
 Then $(H,\rho,\mu)\otimes (H',\rho',\mu')$ in $\bV^{G}(X\otimes X')$  is given by
 $(H\otimes H',\rho\otimes \rho',\mu\otimes \mu')$, where $H\otimes H'$ is the tensor product of Hilbert spaces. In the following we explain how to interpret the  
projection-valued measure $\mu\otimes \mu'$ on all subsets  $X\times X'$ (and not just on subsets of the form $A\times A'$).
Since by assumption
 $H\cong \bigoplus_{x\in X} \mu(\{x\})H$ and $H'\cong \bigoplus_{x'\in X'} \mu'(\{x'\})H'$
 (the sums are interpreted in the sense of Hilbert spaces) we have an isomorphism 
 $H\otimes H'\cong \bigotimes_{(x,x')\in X\times X'} \mu(\{x\})H\otimes \mu'(\{x'\})H'$.
 Using this isomorphism we can define $(\mu\otimes \mu)(A):=\sum_{(x,x')\in A} \mu(\{x\})\otimes \mu'(\{x'\})$, where the sum converges in the  strong topology.  
 
 A morphism  $f\colon (H_{0},\rho_{0},\mu_{0})\to (H_{1},\rho_{1},\mu_{1})$ 
  in $\bV^{G}(X)$ is a bounded operator  $H_{0}\to H_{1}$   which can be approximated uniformly by  equivariant, locally compact and  controlled propagation operators.  If  
  $f'\colon (H'_{0},\rho'_{0},\mu'_{0})\to (H'_{1},\rho'_{1},\mu'_{1})$ is a morphism in $\bV^{G}(X')$, then one checks that 
  $f\otimes f'$ is indeed a morphism  in  $\bV^{G}(X\otimes X')$.

 The unit $\C\to \bV^{G}(*)$ is the functor  which sends the unique object of $\C$ (considered as a $C^{*}$-category with  single object) to $(\C,\mathrm{triv},\nu)$, where $\mathrm{triv}$ is the trivial $G$-representation and $\nu$ is the unique choice of a projection valued measure.
     	\end{proof}

 
Since the tensor unit  $*$ in $G\BC$ is a commutative algebra  object  it follows that $R:=K\cX^{G}(*)$ is a commutative algebra in $\Sp$. 
Moreover, for every object $T$ in $G\Sp\cX$ the $T$-twist $K\cX^{G}_{T}$ of  the functor  $K\cX^{G}$ naturally refines to an equivariant  $\Mod(R)$-valued coarse homology  theory, or equivalently, to 
an element of $\Fun^{\colim}(G\Spc\cX, \Mod(R))$ which  will still be denoted by the same symbol   $K\cX_{T}^{G}$.

In order to calculate  the ring $\pi_{0}(R)$ we observe that
 $\bV^{G}(*)$ is equivalent to  the $C^{*}$-category of finite-dimensional unitary representations of $G$ and equivariant  linear maps. This category is semisimple and generated by the irreducible finite-dimensional unitary representations of $G$. The ring structure on $R$ is induced by the tensor product of finite-dimensional unitary representations of $G$. The representation ring $R(G)$ of $G$ is defined   as the ring-completion of the semiring of isomorphism classes of finite-dimensional unitary representations of $G$.
It then follows that $\pi_{0}(R)\cong R(G)$.

From now on  we assume that $G$ is finite. 

Let $\gamma$ be a conjugacy class of $G$ and $(\gamma)$ in $R(G)$ be the corresponding ideal, see Definition~\ref{rgiogerg43t34t43t}.  Using the isomorphism $\pi_{0}(R)\cong R(G)$ we can form the localizations $R_{(\gamma)}$ of the ring spectrum $R$  and 
$$K\cX^{G}_{T,(\gamma)}:=K\cX^{G}_{T}{\otimes}_{R}R_{(\gamma)}$$
of the $R$-module spectrum valued functor $K\cX_{T}^{G}$.
Recall the Definition \ref{ioujiowrgrgrergergr} of the family of subgroups $F(\gamma)$ and Definition \ref{etgioegggergergergeg} of $\underline{K\cX^{G}_{T,(\gamma)}}$.
 
\begin{prop}\label{rgioefewfewfewf}
$\underline{K\cX^{G}_{T,(\gamma)}}$ vanishes on $\cF(\gamma)$.
\end{prop}
\begin{proof} 
In view of \eqref{g54ljklefwefewfwf} it suffices {to} show that
$$K\cX^{G}_{(\gamma)}((G/K)_{min,max}\otimes X)\simeq 0$$ for every $K$  in $\cF(\gamma)$ and $X$ in $G\BC$. This implies the result by setting $X=T\otimes Y$ for $Y$ in $\BC$.

We must explicitly understand the $R$-module structure of $K\cX^{G}(Z)$ for $Z$ in $G\BC$. It is induced by   the  functor between $C^{*}$-categories 
$$\bV^{G}(*)\otimes  \bV^{G}(Z)\to \bV^{G}(Z)$$
which sends 
$((V,\pi,\nu),(H,\rho,\mu))$ to $(V\otimes H, \pi\otimes \rho,\nu\otimes \mu )$. Here
$(V,\pi,\nu) $ in $\bV^{G}(*)$  is a finite-dimensional unitary representation of $G$ with the canonical measure $\nu$, and $(H,\rho,\mu)$ is in $\bV^{G}(Z)$.
This construction extends to morphisms in the natural way. In particular, for every $(V,\pi)$
we have   a functor between $C^{*}$-categories
$$(V,\pi)\otimes (-) : \bV^{G}((G/K)_{min,max}\otimes X)\to \bV^{G}((G/K)_{min,max}\otimes X)\ .$$

Assume now that $(V^{\pm},\pi^{\pm})$ are two finite-dimensional unitary representations of $G$ such that there exists a unitary isomorphism $W:(V^{+},\pi^{+})_{|K}\to (V^{-},\pi^{-})_{|K}$.
Then we define a $\diag(G/K\times X)$-controlled equivariant unitary isomorphism
$$\tilde W_{(H,\phi,\rho)}:(V^{+}\otimes H, \pi^{+}\otimes \rho, \nu^{+}\otimes \mu)\to (V^{-}\otimes H,  \pi^{-}\otimes \rho,\nu^{-}\otimes \mu)$$ 
by
$$\tilde W_{(H,\phi,\rho)}:=\sum_{gK\in G/K}  \pi^{-}(g)W\pi^{+}(g^{-1})\otimes \mu(\{gK\}\times X)\ .$$
Note that the summands are well-defined since $W$ is $K$-equivariant. 
 This construction is natural in $(H,\rho,\mu)$
  and provides a unitary isomorphism of functors
$$
\xymatrix{ \bV^{G}((G/K)_{min,max}\otimes X) \ar@/^1cm/[rr]^{ (V^{+},\pi^{+})\otimes (-)}\ar@/_1cm/[rr]^{ (V^{-},\pi^{-})\otimes (-)}&\Downarrow \tilde{W} & \bV^{G}((G/K)_{min,max}\otimes X)}\ .  $$
At this point it is crucial that we consider the minimal coarse structure on the factor $G/K$ since this implies that $\tilde W$ induces a natural transformation.

The existence of the isomorphism $\tilde W$ implies that the multiplications by the classes   
$[V^{+},\pi^{+}]$ and $[V^{-},\pi^{-}]$ in $R(G)\cong \pi_{0} K\cX^{G}(*)$ on
$\pi_{*}K\cX^{G}((G/K)_{min,max}\otimes X)$ coincide.
Consequently, if $\rho$ in $R(G)$ satisfies $\rho_{|K}=0$, then $\rho$ acts trivially on
$\pi_{*}K\cX^{G}((G/K)_{min,max}\otimes X)$. 

We now proceed as in the proof of Lemma \ref{wfiowefwefewfewfw}.
By the result of Segal mentioned there we can choose $\eta$ in $R(G)$ such that $\eta_{|K}=0$ and $\Tr(\eta)(g)\not=0$ for all 
$g$ in $\gamma$.  Then $\eta$ is invertible in $\pi_{0} R_{(\gamma)}\cong R(G)_{(\gamma)}$.
By the universal property of the localization the multiplication by $\eta$ on
$K\cX_{(\gamma)}^{G}((G/K)_{min,max}\otimes X)$ is an equivalence, and  by the argument above it vanishes.
Therefore we have $K\cX_{(\gamma)}^{G}((G/K)_{min,max}\otimes X)\simeq 0$.
\end{proof}

%
%

Let $X$ be  a $G$-bornological coarse space and   $\gamma$ be a conjugacy class of $G$.
Let $T$ be an object of $G\Sp\cX$.  
 
\begin{theorem}[The Coarse  Segal Localization Theorem]\label{griuehiuhfergefrggeg} Assume:
\begin{enumerate} \item $G$ is finite.
\item    $\Yo^{s}_{G}(X^{\gamma})\to \Yo^{s}_{G}(\{X^{\gamma}\}) $ is an equivalence.
\item    $\Yo^{s}(X)$ and $\Yo^{s}(X^{\gamma})$ belong to $ G\Sp\cX\langle G\Orb \otimes  \Sp\cX \rangle$.\end{enumerate}
Then the inclusion $X^{\gamma}\to X$ induces an  equivalence
$$K\cX_{T,(\gamma)}^{G}(X^{\gamma})\stackrel{\simeq}{\to} K\cX_{T,(\gamma)}^{G} (X) \ .$$
\end{theorem}
\begin{proof}
Because of our assumptions on $G$ and $X$, using Corollary \ref{erwfiowfwefewfwf}, we can replace
the functor
$K\cX_{T,(\gamma)}^{G}$  by its Bredon-style approximation
$(K\cX_{T,(\gamma)}^{G})^{\Bredon}$.
 By Proposition \ref{rgioefewfewfewf} $\underline{K\cX_{(\gamma)}^{G}}$  vanishes on $\cF(\gamma)$.
In view of Definition \ref{rgfrihgioffwfwfwefewf} we now apply the Theorem \ref{rgioewfwefewfewfewf} in order to get the assertion.
\end{proof}

%
%

\begin{rem}
One can deduce a special case of the  classical  Segal Localization Theorem~\ref{wriowwgwegewgewgewgwas} from Theorem~\ref{griuehiuhfergefrggeg} as follows. We consider finite $G$-simplicial complexes as $G$-uniform bornological coarse spaces with the structures induced from the spherical path metric on the simplices. After passing to a subdivision we can assume that if $g$ in $G$ fixes an interior point  of a simplex, then it fixes the whole simplex.

For a finite $G$-simplicial complex  $W$, essentially by the definition of $\cO^{\infty}_{\hlg}$, we have a natural equivalence \begin{equation}\label{rjf3iuf3f23fwwefwefwef} \Yo^{s}(\cO^{\infty}(W))\simeq \cO^{\infty}_{\hlg}(W)\ ,\end{equation}
where $\cO^{\infty}\colon G\UBC\to G\BC$ is the geometric version of the cone-at-$\infty$ functor (see \cite[Sec.~9]{equicoarse}). The latter sends a $G$-uniform bornological coarse space
$Z$ to the $G$-bornological coarse space $(\R_{d}\otimes Z)_{h}$, where $h$ indicates the hybrid structure on $\R_{d}\otimes Z$ associated to the family of subsets $((-\infty,n]\times Z)_{n\in \nat}$, and $\R_{d}$ has the uniform bornological coarse structures induced from the standard metric.

Combining the description of $KU_{G}$ given in Remark \ref{erthopetrghertge} with the  Paschke duality morphism
from \cite[Thm.~1.4]{bel-paschke} we get an equivalence 
$${\Sigma} KU^{G}_{G}\circ \tilde Y\simeq K\cX^{G}_{G_{can,min}}\circ \cO^{\infty}_{\hlg}$$
of functors from $G\Top$ to $\Sp$. If $W$ is a  finite $G$-simplicial complex, then we  have the equivalences 
$${\Sigma} KU^{G}_{G}(  W)\simeq K\cX_{G_{can,min}}^{G} ( \cO_{\hlg}^{\infty}(W))\simeq K\cX_{G_{can,min}}^{G}( \cO^{\infty}(W))\ ,$$
which are natural in $W$.

 We consider the fixed-point set  {$W^{\gamma}$}  with the uniform bornological coarse structure induced from $W$ (which might be different from the intrinsic one   if $W^{\gamma}$ is not connected).
The fixed point set 
   $W^{\gamma}$ is closed  and a deformation retract of an open neighbourhood. In view of Lemma \ref{gireogoerg34t34erg} we have  a canonical identification  $\cO^{\infty}(W^{\gamma})\cong \cO^{\infty}(W)^{\gamma}$, and
   using the homotopy invariance of $\Yo_{G}^{s}\circ \cO^{\infty}$ \cite[Cor.~9.38]{equicoarse} we furthermore get an equivalence $\Yo^{s}_{G}(\cO^{\infty}(W)^{\gamma})\to \Yo^{s}_{G}(\{\cO^{\infty}(W)^{\gamma}\}) $.

 By Lemma \ref{ergbioeoreggrwrgwergwrgrg} and the equivalence \eqref{rjf3iuf3f23fwwefwefwef}  we know that
 $\cO^{\infty}(W)$ belongs to  $ G\Sp\cX\langle G\Orb \otimes  \Spc\cX \rangle$. The same applies to   $\cO^{\infty}(W^{\gamma})$, since $W^{\gamma}$ is just a subcomplex.
 Therefore we can apply Theorem \ref{griuehiuhfergefrggeg} and conclude:
 \begin{kor}\label{fgwuhwuih23uir23rr}
 The   inclusion $W^{\gamma}\to W$ induces an equivalence
 $$KU^{G}_{G,(\gamma)}( W^{\gamma} )\stackrel{\simeq}{\to} KU^{G}_{G,(\gamma)}( W)\ .$$
   \end{kor}\hrB
\end{rem}

\subsection{Localization for equivariant coarse algebraic $K$-homology}\label{greiuhiufhiwefwefewfewf2r323r1}

In this section we apply the abstract localization results to equivariant coarse algebraic K-theory discussed in Example \ref{rgeroij4oi3ergergregergreg}. {The proofs are the adaptations of the proofs from the previous Section \ref{greiuhiufhiwefwefewfewf2r323r}.}

Let $Q$ be a commutative ring. We consider the additive category $\bQ:=\Mod(Q)^{{fg},free}$  of finitely generated free $Q$-modules.
 The tensor product turns $\bQ$ into a ring object in the symmetric monoidal category of additive categories.
 By definition, a   symmetric monoidal  additive category $\bA$ enriched over $Q$ is a  symmetric monoidal functor $\bQ\to \bA$.

   To every symmetric monoidal additive category $\bA$ {with strict $G$-action} we associate a lax-symmetric monoidal functor
$$\bV^{G}_{\bA}:G\BC\to \Add\ ,$$ see  \cite{symm} for details. The  symmetric monoidal functor $\bQ\to \bA$ gives a  transformation of lax symmetric monoidal  functors \begin{equation}\label{rewgerfercsd}\bV^{G}_{\bQ}\to \bV^{G}_{\bA}\ .  
\end{equation} 
 The algebraic $K$-theory of additive categories can be refined to a lax
  symmetric monoidal functor
$$K:\Add\to \Sp\ .$$ We then get a lax-symmetric monoidal functor (see  again \cite{symm} for details)
$$K\bA\cX^{G}:=K\circ \bV^{G}_{\bA}:G\BC\to \Sp$$ which is an 
 equivariant $\Sp$-valued coarse homology theory.  It is called the equivariant coarse   {algebraic} $K$-homology with coefficients in $\bA$.
 
In particular we get a commutative ring spectrum \begin{equation}\label{wevwlenkwefwefewewfew}R_{\bA}:=K\bA\cX^{G}(*) \end{equation}   and the functor
$K\bA\cX^{G}$ refines to a $\Mod(R_{\bA})$-valued equivariant  coarse homology theory which we will denote by the same symbol.

The symmetric monoidal functor $\bQ\to \bA$
 induces   a natural transformation of lax symmetric monoidal functors
$K\bQ\cX^{G}\to K\bA\cX^{G}$ which induces a ring homomorphism $R_{\bQ}\to R_{\bA}$.

 We note that $\bV^{G}_{\bQ}(*)$ is the symmetric monoidal category of representations of $G$
 on finitely generated free $Q$-modules which can also be written as $\Fun(BG,\bQ)$.
The ring completion of the semi-ring of isomorphism classes in $\bV^{G}_{\bQ}(*)$ will be denoted by $R_{Q}(G)$. 
We have a homomorphism of rings 
$$\kappa:R_{Q}(G)\to \pi_{0} R_{\bQ}\to\pi_{0} R_{\bA}\ .$$
If $K$ is a subgroup of $G$, then restriction   of the action from $G$ to $K$ induces a homomorphism of rings
$$R_{Q}(G)\to R_{Q}(K)\ , \quad  \eta\mapsto \eta_{|K}\ .$$

Let $K$ be a subgroup of $G$ and $X$ be in $G\BC$. Let $\eta$ be an element of $R_{Q}(G)$.
 
\begin{prop}\label{greuihiu4h4u3i4reeg}
If $\eta_{|K}=0$, then the multiplication by $\kappa(\eta)$ on
$\pi_{*}K\bA\cX^{G}((G/K)_{min,max}\otimes X)$ vanishes.
\end{prop}
\begin{proof}
We must explicitly understand the $R_{Q}(G)$-module structure of $\pi_{*}K\bA\cX^{G}(Y)$ for $Y$ in $G\BC$. It is induced by   the {morphism} of additive categories 
$$\bV_{\bQ}^{G}(*)\otimes  \bV_{\bA}^{G}(Y)\to \bV_{\bA}^{G}(Y)$$
which sends  
$((V,\pi),(H,\rho))$ to $(V\otimes H,\pi\otimes \rho)$. Here
$(V,\pi)$ is a    representation of $G$ on a finitely generated free $Q$-module, and $(H,\rho)$ is an equivariant $Y$-controlled $\bA$-object, i.e,    an equivariant cosheaf (determined on points) on the bornology $\cB_{Y}$ of $Y$ with values in $\bA$ \cite[Def.~8.3]{equicoarse}. 
With the obvious extension of the constructions above to morphisms we have  actually constructed an exact functor between additive categories 
$$(V,\pi)\otimes (-) :\bV_{\bA}^{G}((G/K)_{min,max}\otimes X)\to \bV_{\bA}^{G}((G/K)_{min,max}\otimes X)\ .$$
Assume now that $(V^{\pm},\pi^{\pm})$ are two    representation of $G$ on finitely generated projective $Q$-modules  such that there exists an  isomorphism $W:(V^{+},\pi^{+})_{|K}\to (V^{-},\pi^{-})_{|K}$. Then we define a $\diag(G/K\times X)$-controlled isomorphism
$$\tilde W_{(H,\rho)}:(V^{+}\otimes H,\pi^{+}\otimes \rho)\to (V^{-}\otimes H,\pi^{-}\otimes \rho)$$
by  
$$(\tilde W_{(H,\rho)})_{(gK,x)}:=\pi^{-}(g)W\pi^{+}(g^{-1})\otimes \id_{H(\{(gK,x)\})}:V^{+}\otimes H(\{(gK,x)\})\to V^{-}\otimes H(\{(gK,x)\})$$
for all points $(gK,x)$ in $G/K\times X$.
Since we consider the minimal coarse structure on the $G/K$-factor these isomorphisms provide a natural isomorphism of functors
$$
\xymatrix{ \bV_{\bA}^{G}((G/K)_{min,max}\otimes X) \ar@/^1cm/[rr]^{ (V^{+},\pi^{+})\otimes (-)}\ar@/_1cm/[rr]^{ (V^{-},\pi^{-})\otimes (-)}&\Downarrow  \tilde W & \bV_{\bA}^{G}((G/K)_{min,max}\otimes X)}\ .  $$
The existence of the isomorphism $\tilde W$ implies that the multiplications by the classes   
$\kappa([V^{+},\pi^{+}])$ and $\kappa([V^{-},\pi^{-}])$  on
$\pi_{*}K\bA\cX^{G}((G/K)_{min,max}\otimes X)$ coincide.
Consequently, if $\eta$ in $R_{Q}(G)$ satisfies $\eta_{|K}=0$, then $\eta$ acts trivially on
$\pi_{*}K\bA\cX^{G}((G/K)_{min,max}\otimes X)$. 
\end{proof}

 {Recall Definition \ref{rgioregggegergergerg} and Definition \ref{zfikgloglugilgilk}.}
Let $Q$ be a commutative ring and $\bA$ be a symmetric monoidal additive
 category enriched over $Q$. We  fix a conjugacy class $\gamma$ in $G$ and consider a prime ideal $I$ in the commutative ring $R_{Q}(G)$. {Let $T$ be  an object in $G\Sp\cX$.}
\begin{theorem}\label{giehrighiuhfiwebfjwefwefewfewf}
 Assume
 \begin{enumerate}\item \label{greighiuhiweergreg} One of 
 \begin{enumerate}
 \item \label{oegjoi34jergegerg} $G$ is finite and $X^{\gamma},X$ {belong} to $G\Sp\langle G\Orb\otimes \Sp\cX\rangle$.
 \item \label{oegjoi34jergegerg1}  $T{\simeq  \Yo^{s}_{G}(G_{can,min})}$ and $X^{\gamma},X$ {belong} to $G\Sp\langle G\Orb\otimes \Sp \cX_{bd}\rangle$.
 \end{enumerate}
 \item\label{gerihgiu4hiufergg} For every $K$ in $F(\gamma)$ there exists $\eta$ in $R_{Q}(G) $ with $ \eta_{|K}=0$ and which is invertible in $R_{Q}(G)_{(I)}$.
 \item $\Yo^{s}_{G}(X^{\gamma})\to \Yo^{s}_{G}(\{X^{\gamma}\}) $ is an equivalence.

\end{enumerate}
 Then the inclusion $X^{\gamma}\to X$ induces an equivalence
 $$K\bA\cX_{T,(I)}^{G}(X^{\gamma})\to K\bA\cX_{T,(I)}^{G}(X )\ .$$
 \end{theorem}
\begin{proof}
By Assumption \ref{greighiuhiweergreg} we can replace $K\bA\cX^{G}_{T,(I)}$ by its Bredon-style 
approximation. In Case  \ref{greighiuhiweergreg}.\ref{oegjoi34jergegerg} we use 
Corollary \ref{erwfiowfwefewfwf}, and in  Case \ref{greighiuhiweergreg}.\ref{oegjoi34jergegerg1}
we use Corollary \ref{erwfiowfwefewfwf1}, the fact that $K\bA\cX^{G}$ is continuous  \cite[Prop.~8.17]{equicoarse}, and Lemma \ref{fewfpowef23r2} implying that $K\bA\cX^{G}_{G_{can,min},(I)}$ is hyperexcisive.
Using Assumption \ref{gerihgiu4hiufergg} we show as in the proof of Proposition~{\ref{rgioefewfewfewf}} and using Proposition \ref{greuihiu4h4u3i4reeg} that
$\underline{K\bA\cX^{G}_{T,(I)}}$ vanishes on $F(\gamma)$.
The asserted equivalence now follows from Corollary \ref{gerhgiu34hiu34huferererg}.
\end{proof}

\begin{rem}
In general the Assumption \ref{giehrighiuhfiwebfjwefwefewfewf}.\ref{gerihgiu4hiufergg} is strong.
If $G$ is finite and $Q=\C$, 
 then it is satisfied  for the ideal $I=(\gamma)$ introduced in Definition \ref{rgiogerg43t34t43t} by the result of Segal~\cite{segal} mentioned in the proof of Lemma \ref{wfiowefwefewfewfw}. We therefore get:
\begin{theorem}\label{gerkhihi3434g4aa}
 Assume {that} $\bA$ is a symmetric monoidal additive category enriched over~$\C$ and:  
 \begin{enumerate}\item $G$ is finite. 
 \item \label{greighiuhiweergreg11aa} 
$X^{\gamma},X$ belong to $G\Sp\langle G\Orb\otimes \Sp\cX\rangle$.
 \item  $\Yo^{s}_{G}(X^{\gamma})\to \Yo^{s}_{G}(\{X^{\gamma}\}) $ is an equivalence.

\end{enumerate}
 Then the inclusion $X^{\gamma}\to X$ induces an equivalence
 $$K\bA\cX_{T,(\gamma)}^{G}(X^{\gamma})\to K\bA\cX_{T,(\gamma)}^{G}(X )\ .$$
 \end{theorem}\hrB
\end{rem}

\begin{rem}
In  Theorems \ref{giehrighiuhfiwebfjwefwefewfewf} and \ref{gerkhihi3434g4aa} the assumption that $\bA$ is a symmetric monoidal category enriched over $Q$ can be relaxed. It suffices to assume that  $\bA$ is just enriched over $Q$, i.e., that $\bA$ is an additive module category over $\bQ$. Then
$\bV^{G}_{\bA}$ becomes a module over $\bV^{G}_{\bQ}$, and this suffices for the arguments.
We worked with the stronger assumption in order to be able to cite \cite{symm} directly in order
to get the  transformation of lax symmetric monoidal  functors \eqref{rewgerfercsd}. In order to formally
deduce the more general case from the special case considered here one could encode the
$\bQ$-module structure on $\bA$ in a square-zero extension of symmetric monoidal additive categories
$\bQ\to \bQ\oplus \bA$. We leave the details to the interested reader. \hrB
 \end{rem}

\begin{rem}
The localization theorems for {the}  equivariant coarse algebraic  $K$-homology {theory} $K\bA\cX^{G}$ implies a localization theorem for the equivariant homology theory   $K\bA^{G}:=K\bA\cX\circ \cO^{\infty}_{\hlg}$ {which is} similar to Corollary \ref{fgwuhwuih23uir23rr}.  We again leave the details to the interested reader.\hrB
\end{rem}

\bibliographystyle{alpha}

\end{document}